\documentclass[12pt]{amsart}
\usepackage[a4paper,margin=1in]{geometry}
\usepackage{amsmath,amssymb,amsthm,mathtools}
\usepackage{bm}
\usepackage{graphicx}
\usepackage{xcolor}
\usepackage{float}
\usepackage{tikz}
\usepackage{hyperref}
\usepackage{enumitem}
\usepackage{booktabs}
\usepackage{cite}
\hypersetup{
	hidelinks,
	pdftitle={On hot spots of the torsion function and the ground state of the Laplacian},
	pdfauthor={Jiaxin He, Qinfeng Li, Shangfeng Wu, and Shuangquan Xie}
}

\newtheorem{theorem}{Theorem}[section]
\newtheorem{proposition}[theorem]{Proposition}
\newtheorem{lemma}[theorem]{Lemma}
\newtheorem{corollary}[theorem]{Corollary}
\newtheorem{conjecture}[theorem]{Conjecture}
\theoremstyle{definition}

\newtheorem{remark}[theorem]{Remark}
\newtheorem{numericalobservation}[theorem]{Numerical Observation}

\title[Hot spots of torsion function and ground state]{On hot spots of the torsion function and the ground state of the Laplacian}
\author[J. He]{Jiaxin He}
\address{School of Mathematics, Hunan University, Changsha, P.R. China.}
\email{17773612125@163.com}
\author[Q. Li]{Qinfeng Li}
\address{School of Mathematics, Hunan University, Changsha, P.R. China.}
\email{liqinfeng1989@gmail.com}
\author[S. Wu]{Shangfeng Wu}
\address{School of Mathematics, Hunan University, Changsha, P.R. China.}
\email{202621130048@mail.bnu.edu.cn}
\author[S. Xie]{Shuangquan Xie}
\address{School of Mathematics, Hunan University, Changsha, P.R. China.}
\email{xieshuangquan@hnu.edu.cn‌}
\date{}

\begin{document}
	\begin{abstract}
		We compare the maximum points \(P_\Omega\) and \(Q_\Omega\) of the
		torsion function and the positive first Dirichlet eigenfunction. We begin with a systematic analytic description of both
		hot spots over the full range of opening angles of circular sectors: uniqueness,
		strict monotonicity with respect to the angle, distinct
		small-angle asymptotics, and explicit limiting equations in the
		slit-disk limit. We also prove the existence of a coincidence angle and establish arithmetic
		properties of the limiting coordinates. We then consider bounded convex planar domains with a line of symmetry, and we prove the sharp diameter-normalized identity
		\[
		\sup_\Omega
		\frac{|P_\Omega-Q_\Omega|}{\operatorname{diam}(\Omega)}=\frac12.
		\]
		Its upper-bound argument uses standard localization in the heart and
		an elementary geometric enclosure. Sharpness follows from thin convex
		domains with a vanishing relative width perturbation peaking near one
		endpoint: the torsion hot spot approaches that endpoint, whereas the
		ground-state hot spot approaches the midpoint.
	\end{abstract}
	\maketitle
	
	\section{Introduction}
	
	\noindent The paper studies two aspects of the comparison between the torsion
	hot spot and the ground-state hot spot: how the two points move with the
	opening angle in a circular sector
	(Sections~\ref{sec:sector-torsion}--\ref{sec:sector-comparison}), and how far
	apart they can be in a convex planar domain with a line of symmetry
	(Section~\ref{sec:convex-diameter}). The sharp value of this
	diameter-normalised separation is \(1/2\). Section~\ref{sec:triangles} treats
	sharp isosceles triangles. The two parts are proved independently, and a
	reader interested only in the sharp constant may begin at
	Section~\ref{sec:convex-diameter}.
	
	\subsection{Motivations and our results on circular sectors}
	Let \(\Omega\subset \mathbb R^2\) be a bounded planar domain. The torsion function of \(\Omega\) is the solution of
	\[
	-\Delta u=1\quad \hbox{in }\Omega,
	\qquad u=0\quad \hbox{on }\partial\Omega .
	\]
	A maximum point of \(u\) will be called a torsion hot spot. This terminology is meant literally in the standard heat interpretation: \(u\) is the steady-state temperature in a body with constant heat source and zero boundary temperature, so the torsion hot spot is the location of maximal temperature.  The same function also has classical mechanical and fluid interpretations.  In Saint-Venant torsion theory, \(|\nabla u|\) represents the magnitude of the shear stress, and the maximum of \(u\) is a natural geometric marker of the torsion problem. If \(u\) is viewed as a stream function, then \(\nabla^\perp u\) is a steady incompressible Euler flow; the level sets of \(u\) are streamlines and the critical points of \(u\) are stagnation points.  Thus the maximum point of \(u\) is both the hottest point in the heat model and a distinguished stagnation point in the associated steady flow. These should not be confused with the hot spots of the Neumann problem for the heat equation, where the extrema of the second Neumann eigenfunction are at issue; every maximum discussed here is an interior maximum of a Dirichlet problem.
	
	\begingroup
	A further motivation comes from work of the second author with Yao
	\cite{liyao2024} and with Xie, Yang and Yao \cite{lixieyangyao2026} on fail
	points, the boundary points where \(|\nabla u|\) attains its maximum. For a
	convex planar domain symmetric with respect to both coordinate axes,
	Saint-Venant's conjecture asserts that they occur at the contact points of a
	largest inscribed circle centered at the symmetry center, which is also the
	torsion hot spot. Although the conjecture fails in general
	\cite{liyao2024}, the question suggests a link between interior torsion
	maxima and boundary stress maxima, and motivates locating the torsion hot
	spot precisely.
	\par\endgroup
	
	When \(\Omega\) is convex, the torsion hot spot is unique by the classical concavity theory of Makar-Limanov, Brascamp--Lieb, Korevaar, Kennington, and related work \cite{makarlimanov,brascamplieb,korevaar1983,kennington,korevaarlewis}.
	
	The second object considered in this paper is the positive first Dirichlet eigenfunction
	\[
	-\Delta \phi_1=\lambda_1(\Omega)\phi_1\quad \hbox{in }\Omega,
	\qquad
	\phi_1=0\quad \hbox{on }\partial\Omega,
	\qquad
	\phi_1>0.
	\]
	We call its maximum point the ground-state hot spot. In convex domains,
	\(\phi_1\) is log-concave by the Brascamp--Lieb theorem
	\cite{brascamplieb}. Its maximum point is unique, as follows from
	log-concavity together with interior analyticity. Indeed, two distinct
	maximum points would force \(\phi_1\) to equal its positive maximum
	on the segment joining them. The restriction of \(\phi_1\) to the
	line through that segment is real analytic on its intersection with
	\(\Omega\), which is an interval by convexity. The one-dimensional
	identity theorem would then make this restriction constant on the
	whole interval, contradicting the zero boundary values at its
	endpoints. In the convex case considered below, we denote the two points by
	\[
	P_\Omega=\operatorname*{arg\,max}_\Omega u,
	\qquad
	Q_\Omega=\operatorname*{arg\,max}_\Omega \phi_1.
	\]
	
	A geometric framework for locating these points is provided by the
	\emph{heart} of a convex domain, introduced by
	Brasco--Magnanini--Salani \cite{brascomagnaninisalani2011} and
	studied further by Brasco--Magnanini \cite{brascomagnanini2013}.
	This set is determined by the domain through reflection arguments
	and gives a common enclosure for maxima of positive solutions of
	a broad class of autonomous Dirichlet equations; in particular,
	it contains both \(P_\Omega\) and \(Q_\Omega\).
	Such a common geometric localization does not by itself determine
	the relative positions selected by two different equations within
	that enclosure. The present paper studies this question for two
	canonical problems, the torsion equation and the first eigenvalue
	equation of the Dirichlet Laplacian.

	Benson--Laugesen--Minion--Siudeja \cite{bensonlaugesen} showed that
	the torsion and ground-state hot spots are close but distinct in the
	half disk and the isosceles right triangle.
	For circular sectors, their half-disk example leaves open whether this
	close-but-not-identical behavior persists for every opening angle.
	Is the half disk representative of all sectors, or can the relative
	order of the two points change? Can the two hot spots coincide for
	some opening angle?
	These are the starting questions of our sector analysis.
	
	We study all unit sectors \(0<\alpha<2\pi\), including the nonconvex
	ones. In addition to comparing the two points, we determine how each
	location varies with the opening angle and describe its behavior as
	\(\alpha\downarrow0\) and \(\alpha\uparrow2\pi\).
	The next theorem collects the results of Sections~\ref{sec:sector-torsion}--\ref{sec:sector-comparison}: both hot spots are unique and lie on the axis, each moves strictly toward the vertex as the sector opens, their small-angle behaviour is governed by two different scales, and the slit-disk limit is described by two scalar equations.
	
	\begin{theorem}
		\label{thm:sector-summary}
		Let
		\[
		S(\alpha)=\{(r,\theta):0<r<1,\ |\theta|<\alpha/2\},
		\qquad 0<\alpha<2\pi.
		\]
		Then the torsion function and the positive first Dirichlet eigenfunction
		have unique maximum points
		\[
		P_{S(\alpha)}=(x_\alpha,0),
		\qquad Q_{S(\alpha)}=(y_\alpha,0)
		\]
		in polar coordinates. The following properties hold.
		\begin{enumerate}[label=\textup{(\roman*)},leftmargin=*]
			\item
			Both \(x_\alpha\) and \(y_\alpha\) are continuous and strictly decreasing
			in \(\alpha\). With \(\nu=\pi/\alpha\),
			\[
			y_\alpha=\frac{j'_{\nu,1}}{j_{\nu,1}},
			\]
			where \(j_{\nu,1}\) and \(j'_{\nu,1}\) are the first positive zeros
			of \(J_\nu\) and \(J_\nu'\), respectively.
			\item
			Let \(\beta=2\alpha/\pi\), and let \(a_1,a'_1\) be the first negative
			zeros of \(\operatorname{Ai}\) and \(\operatorname{Ai}'\). As
			\(\alpha\downarrow0\),
			\[
			\begin{aligned}
				1-x_\alpha
				&=\frac{\beta}{2}\log\frac1\beta
				-\frac{\beta}{2}\log\frac{\pi^3}{32}
				+O\!\left(\beta^2\log^2\frac1\beta\right),\\
				1-y_\alpha
				&=\frac{|a_1|-|a'_1|}{2}\,\beta^{2/3}
				+O(\beta^{4/3}).
			\end{aligned}
			\]
			In particular, \(x_\alpha>y_\alpha\) for all sufficiently small
			\(\alpha>0\).
			\item
			In the half disk, \(x_\pi<y_\pi\). Consequently, there exists at least
			one \(\alpha_0\in(0,\pi)\) for which \(x_{\alpha_0}=y_{\alpha_0}\).
			\item
			As \(\alpha\uparrow2\pi\),
			\[
			x_\alpha\longrightarrow x_*=t_*^2,
			\qquad
			y_\alpha\longrightarrow y_*=\frac{z_*}{\pi},
			\]
			where \(z_*\in(0,\pi/2)\) and \(t_*\in(0,1)\) are the unique solutions of
			\[
			\tan z_*=2z_*,
			\qquad
			3(1-t_*^8)\arctan t_*
			=t_*(3t_*^6-t_*^4-t_*^2+3).
			\]
			The ranges of the two coordinates are \((x_*,1)\) and \((y_*,1)\).
			The numbers \(x_*\) and \(\pi y_*\) are transcendental, and \(y_*\)
			is irrational.
		\end{enumerate}
		The limiting domain in \textup{(iv)} is the unit disk with the negative
		radial segment removed; the two faces of the slit retain Dirichlet
		boundary conditions.
	\end{theorem}
	
	Figure~\ref{fig:sector} locates the two points in a sector.
	
	\begin{figure}[htbp]
		\centering
		\begin{tikzpicture}[scale=3.6]
			\fill[black!7] (0,0) -- (40:1) arc (40:-40:1) -- cycle;
			\draw[thick] (0,0) -- (40:1) arc (40:-40:1) -- cycle;
			\draw[dashed,black!55] (-0.05,0) -- (1.06,0);
			\draw (0.26,0) arc (0:40:0.26);
			\node at (0.33,0.12) {\scriptsize $\alpha/2$};
			\node[circle,fill,inner sep=0.55pt] at (0.56,0) {};
			\node[circle,fill,inner sep=0.55pt] at (0.70,0) {};
			\node[below=6pt] at (0.56,0) {\scriptsize $P_{S(\alpha)}$};
			\node[above=5pt] at (0.70,0) {\scriptsize $Q_{S(\alpha)}$};
			\node[below left] at (0,0) {\scriptsize $0$};
		\end{tikzpicture}
		\caption{The sector $S(\alpha)$ and the two hot spots, both of which lie on the axis.
			Their separation is exaggerated. The largest separation observed in the
			numerical exploration is approximately $0.00712442$; a global numerical
			maximum is not proved here.}
		\label{fig:sector}
	\end{figure}
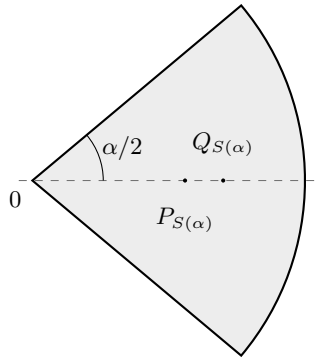

	For torsion, the main analytic device is a positive integral
	representation of
	\(H_\alpha(r)=2T_\alpha(r)-rT_\alpha'(r)\), where
	\(T_\alpha(r)=u_\alpha(r,0)\).
	It yields uniqueness of the axial critical point and permits a
	comparison as the angle varies. For the ground state, the explicit
	Bessel representation identifies the unique maximum, while a
	Sturm--Liouville comparison gives monotonicity.
	Thus the sector analysis describes both locations throughout the angle
	range, even though their difference need not inherit their individual
	monotonicity. Numerical exploration suggests a largest separation of
	about \(0.0071244160\); global numerical extremality and uniqueness
	of a coincidence angle are not asserted in the theorem.
	
	The two parts of the paper answer different questions, but they rest on one
	observation: the two maxima react to a change of width at different rates,
	so a feature that is negligible for one equation can decide the other. In
	the sharp-sector limit this produces two distinct scales,
	\(\beta\log(1/\beta)\) for the torsion maximum and \(\beta^{2/3}\) for the
	ground-state maximum, which is the source of the change of order described
	above. The construction of Section~\ref{sec:convex-diameter} does not
	repeat that asymptotic analysis; it exploits the same contrast in a different
	regime, where the torsion profile is almost flat and the ground-state
	profile is a sine whose maximum is nondegenerate. The sector family is where that contrast
	can be computed explicitly.
	
	\vskip 0.3cm
	
	\subsection{Main result: sharp bound for the distance}
	The results on circular sectors lead to the second, more global question of the paper. What is the optimal diameter-normalised upper bound for the distance between the torsion hot spot and the ground-state hot spot among bounded convex planar domains with a line of symmetry?
	
	That the two points can be separated at the scale of the diameter is not
	new. Beck \cite[Proposition~2.4]{beck2020} constructed convex domains of high
	eccentricity whose two maxima are at a distance comparable to the diameter, so the
	quantity we are about to bound is not small in general, and no argument may
	presuppose that the two maxima are close. What is not known is its exact value. Symmetry
	does not by itself prevent a large separation --- Theorem~\ref{thm:main}
	below shows that the supremum of the ratio over the symmetric class is
	\(1/2\) --- but it
	makes the upper bound accessible, because the heart of a symmetric domain
	collapses onto a single interval of the axis.
	
	Since the distance scales linearly with the domain, the natural normalization is by the diameter.  We define
	\[
	M_{\rm sym}:=
	\sup_{\substack{\Omega\subset\mathbb R^2\ \text{bounded and convex}\\
			\Omega\ \text{symmetric with respect to a line}}}
	\frac{|P_\Omega-Q_\Omega|}{\operatorname{diam}(\Omega)}.
	\]
	The main result of the paper is that the sharp answer is exactly \(1/2\).
	
	\begin{theorem}\label{thm:main}
		One has
		\[
		M_{\rm sym}=\frac12.
		\]
		Moreover the supremum is not attained: for every bounded convex planar
		domain \(\Omega\) with a line of symmetry,
		\[
		|P_\Omega-Q_\Omega|<\frac12\operatorname{diam}(\Omega).
		\]
	\end{theorem}
	
	The symmetry assumption is essential to the present proof.  If
	\[
	M:=
	\sup_{\substack{\Omega\subset\mathbb R^2\ \mathrm{bounded\ and\ convex}}}
	\frac{|P_\Omega-Q_\Omega|}{\operatorname{diam}(\Omega)},
	\]
	then the sharpness construction in Theorem~\ref{thm:main}
	gives \(M\ge 1/2\), but the reverse inequality is not proved here. We expect the same bound without any symmetry assumption.
	
	\begin{conjecture}\label{conj:general-convex}
		For every bounded planar convex domain \(\Omega\),
		\[
		|P_\Omega-Q_\Omega|
		\le \frac12\operatorname{diam}(\Omega).
		\]
	\end{conjecture}

	The upper bound follows from standard localization in the heart,
	followed by a short geometric estimate.
	Writing a symmetric convex domain as
	\[
	\Omega=\{(x,y):0<x<L,\ |y|<h(x)\},
	\qquad \operatorname*{arg\,max}_{[0,L]}h=[a,b],
	\]
	we show that its heart is contained in
	\[
	\left[\frac b2,\frac{L+a}{2}\right]\times\{0\}.
	\]
	This gives a common interval of length \((L-(b-a))/2\) for the two
	hot spots. The strict inequality for a fixed domain follows from
	one endpoint reflection comparison; the details are included in
	Proposition~\ref{prop:upperbound}.

	Sharpness is obtained from a family of very thin symmetric convex domains
	\[
	\Omega_\varepsilon=\{(x,y):0<x<1,\ |y|<\varepsilon f_\varepsilon(x)\},
	\]
	where \(f_\varepsilon\) is almost constant but has a concave bump of relative
	size \(\varepsilon^{2+\sigma}\), for any fixed \(\sigma>0\), near the left
	endpoint.  The first eigenfunction keeps the longitudinal profile of the
	rectangle: after the anisotropic scaling \(y=\varepsilon Y\) the profile
	converges to \(\sin(\pi x)\), whose peak at \(x=1/2\) is nondegenerate, so
	the ground-state hot spot still tends to \((1/2,0)\).  The torsion function,
	by contrast, senses the local width: on the axis it behaves like the square
	of the half-width, so the bump forces its maximum toward the left endpoint.
	Thus
	\[
	P_{\Omega_\varepsilon}\to (0,0),
	\qquad
	Q_{\Omega_\varepsilon}\to (\frac{1}{2},0),
	\]
	while \(\operatorname{diam}(\Omega_\varepsilon)\to1\), proving sharpness.
	
	\vskip 0.3cm
	\subsection{Other results}
	
	We also discuss a related sharp-domain model: isosceles triangles
	\[
	\Omega_h=\{(x,y):0<x<1,\ |y|<hx\}.
	\]
	For sufficiently small \(h>0\), we prove that the torsion hot spot lies closer to the base than the ground-state hot spot.  The proof again compares two different scales.  The torsion hot spot is within \(O(h\log(1/h))\) of the base, while the ground-state hot spot is at distance asymptotic to a positive multiple of \(h^{2/3}\).  This parallels the small-angle sector calculation and gives another example in which torsion and ground-state maxima respond differently to sharp geometry.

	The same localization is available for \(-\Delta u=f(u)\) with homogeneous Dirichlet conditions, and we expect the enclosure argument to extend to monotone nonlinearities; we do not pursue this here.

	\section{Sectorial torsion hot spots: expansion and uniqueness}
	\label{sec:sector-torsion}
	
	Throughout this section,
	\[
	S(\alpha)=\{(r,\theta):0<r<1,\ |\theta|<\frac{\alpha}{2}\},
	\qquad 0<\alpha<2\pi.
	\]
	Set
	\[
	\nu=\frac{\pi}{\alpha},
	\qquad
	\mu_k=(2k+1)\nu,
	\qquad k=0,1,2,\ldots .
	\]
	
	The even Dirichlet modes \(\cos(\mu_k\theta)\) are complete in the even subspace of \(L^2(-\alpha/2,\alpha/2)\), so the torsion function expands in them, with radial profiles that can be computed one mode at a time.
	
	\begin{proposition}
		The torsion function in \(S(\alpha)\) is
		\begin{equation}\label{eq:torsion-series}
			u_\alpha(r,\theta)
			=
			\sum_{k=0}^{\infty}
			\frac{4(-1)^k}{(2k+1)\pi}
			q_{\mu_k}(r)\cos(\mu_k\theta),
		\end{equation}
		where
		\begin{equation}\label{eq:qmu}
			q_\mu(r)=\frac{r^2-r^\mu}{\mu^2-4},
			\qquad \mu\ne2,
		\end{equation}
		and
		\begin{equation}\label{eq:q2}
			q_2(r)=-\frac14 r^2\log r.
		\end{equation}
	\end{proposition}
	
	\begin{proof}
		For each fixed \(0<r<1\), the angular profile
		\(u_\alpha(r,\cdot)\) vanishes at \(\theta=\pm\alpha/2\)
		and is even by symmetry and uniqueness of the Dirichlet solution.
		The even Dirichlet modes
		\[
		\cos(\mu_k\theta),\qquad
		\mu_k=\frac{(2k+1)\pi}{\alpha},\qquad k=0,1,2,\ldots,
		\]
		form a complete orthogonal system in the even subspace of
		\(L^2(-\alpha/2,\alpha/2)\). Thus we may expand
		\[
		u_\alpha(r,\theta)=\sum_{k=0}^{\infty}R_k(r)\cos(\mu_k\theta),
		\qquad
		1=\sum_{k=0}^{\infty}a_k\cos(\mu_k\theta),
		\]
		with both expansions understood in
		\(L^2(-\alpha/2,\alpha/2)\), where
		\[
		a_k=
		\frac{\displaystyle\int_{-\alpha/2}^{\alpha/2}\cos(\mu_k\theta)\,d\theta}
		{\displaystyle\int_{-\alpha/2}^{\alpha/2}\cos^2(\mu_k\theta)\,d\theta}
		=\frac{4(-1)^k}{(2k+1)\pi}.
		\]
		Projecting the equation \(-\Delta u_\alpha=1\) onto each angular
		mode gives
		\[
		-R_k''-\frac1rR_k'+\frac{\mu_k^2}{r^2}R_k=a_k,
		\qquad R_k(1)=0.
		\]
		The corresponding homogeneous equation has solutions
		\(r^{\mu_k}\) and \(r^{-\mu_k}\), and boundedness at the vertex
		excludes the latter. For \(\mu_k\ne2\), a particular solution is
		\(\frac{a_kr^2}{\mu_k^2-4}\), so
		\[
		R_k(r)=\frac{a_k}{\mu_k^2-4}r^2+C_kr^{\mu_k}.
		\]
		The boundary condition at \(r=1\) yields
		\(C_k=-\frac{a_k}{\mu_k^2-4}\), and hence \(R_k=a_kq_{\mu_k}\).
		When \(\mu_k=2\), direct substitution and the same boundary
		conditions give
		\[
		R_k(r)=-\frac14a_kr^2\log r=a_kq_2(r).
		\]
		
		For fixed $\alpha$, the summands are bounded in absolute value by
		$\frac{C_\alpha}{(k+1)^3}$, so the series converges uniformly on the closed
		sector, with each summand extended by zero at the vertex.
		Its continuous sum agrees with $u_\alpha(r,\cdot)$ in
		$L^2(-\alpha/2,\alpha/2)$ for every $r\in(0,1)$, and hence
		pointwise. This proves \eqref{eq:torsion-series}.
	\end{proof}
	
	Restricting the expansion to the axis reduces the problem to a one-dimensional equation for \(T_\alpha\).
	
	\begin{corollary}\label{cor:axis-equation}
		Let
		\[
		T_\alpha(r)=u_\alpha(r,0).
		\]
		Then
		\[
		T_\alpha(r)
		=
		\sum_{k=0}^{\infty}
		\frac{4(-1)^k}{(2k+1)\pi}
		\frac{r^2-r^{\mu_k}}{\mu_k^2-4},
		\]
		where, in this formula and its derivative below, a term with
		\(\mu_k=2\) is understood by using \(q_2\) from \eqref{eq:q2} and its
		derivative.  The axial critical points are precisely the roots in
		\((0,1)\) of
		\begin{equation}\label{eq:x-alpha-equation}
			\sum_{k=0}^{\infty}
			\frac{4(-1)^k}{(2k+1)\pi}
			\frac{2r-\mu_k r^{\mu_k-1}}{\mu_k^2-4}=0,
		\end{equation}
	\end{corollary}
	
	\begin{proof}
		Substituting \(\theta=0\) in \eqref{eq:torsion-series} gives the
		formula for \(T_\alpha\). For every \([\rho,R]\Subset(0,1)\) and
		all sufficiently large \(k\),
		\[
		\sup_{r\in[\rho,R]}|a_kq_{\mu_k}'(r)|
		\le C_{\alpha,\rho,R}\bigl(k^{-3}+k^{-2}R^{\mu_k}\bigr).
		\]
		Thus the derivative series converges uniformly on compact
		subintervals of \((0,1)\), and termwise differentiation is valid.
		Symmetry gives \(\partial_\theta u_\alpha(r,0)=0\), so an axial
		point is critical precisely when \(T_\alpha'(r)=0\).
	\end{proof}
	
	For the half disk the axial series can be summed explicitly. The closed formula is used in Section~\ref{sec:sector-comparison} and also serves as a check on the expansion itself.
	
	\begin{corollary}\label{cor:half-disk}
		When \(\alpha=\pi\),
		\[
		S(\pi)=\{(x,y):x^2+y^2<1,\ x>0\},
		\]
		and the torsion function is
		\begin{align}\label{eq:half-disk-closed}
			u(x,y)=\frac1{4\pi}\bigg[&
			-2\pi x^2
			-2x\bigl((x^2+y^2)^{-1}-1\bigr)\notag\\
			&+\Bigl(2+(x^2-y^2)\bigl((x^2+y^2)^{-2}+1\bigr)\Bigr)
			\arctan\frac{2x}{1-(x^2+y^2)}\notag\\
			&+xy\bigl((x^2+y^2)^{-2}-1\bigr)
			\log\frac{x^2+(1+y)^2}{x^2+(1-y)^2}\bigg].
		\end{align}
		At the origin, the right-hand side is understood by continuity.
	\end{corollary}
	
	\begin{proof}
		The proof is given in Appendix~\ref{app:half-disk}.
	\end{proof}
	
	The next theorem demonstrates the uniqueness of hot spots of the torsion function in any sectorial domains.
	\begin{theorem}
		\label{thm:unique}
		For every \(0<\alpha<2\pi\), the torsion function
		\(u_\alpha\) is even in \(\theta\), and for each fixed \(0<r<1\),
		the function \(\theta\mapsto u_\alpha(r,\theta)\) is strictly decreasing
		on \([0,\alpha/2)\).  Moreover, the axial function
		\(T_\alpha(r)=u_\alpha(r,0)\) has a unique critical point
		\(x_\alpha\in(0,1)\), and
		\[
		T_\alpha'(r)>0\quad(0<r<x_\alpha),
		\qquad
		T_\alpha'(r)<0\quad(x_\alpha<r<1),
		\]
		with \(T_\alpha''(x_\alpha)<0\).  Consequently,
		\((x_\alpha,0)\) is the unique, nondegenerate maximum point of \(u_\alpha\) in
		\(S(\alpha)\).
	\end{theorem}
	
	\begin{proof}
		Symmetry and uniqueness give \(u_\alpha(r,-\theta)=u_\alpha(r,\theta)\).
		Set \(v=-\partial_\theta u_\alpha\) in the upper half-sector
		\(0<r<1\), \(0<\theta<\alpha/2\). Since angular differentiation
		commutes with the Laplacian, \(\Delta v=0\). Symmetry gives
		\(v=0\) on \(\theta=0\), and the zero boundary value on the circular
		arc gives \(v=0\) on \(r=1\). On the straight side
		\(\theta=\alpha/2\), Hopf's lemma gives
		\[
		\frac1r\partial_\theta u_\alpha
		=\partial_nu_\alpha<0,
		\qquad\text{hence }v>0\quad(0<r<1).
		\]
		The terms of the angular derivative series are bounded by
		$C_\alpha(k+1)^{-2}$, so the series converges uniformly on the closed sector. Thus \(v\) extends continuously at the vertex and the endpoints of the arc with value zero. The maximum
		principle and its strong form now yield \(v>0\) throughout the upper
		half-sector. Consequently, \(u_\alpha(r,\theta)\) is strictly decreasing
		in \(\theta\in[0,\alpha/2)\), and every maximum lies on the axis.
		
		Write \(T_\alpha(r)=u_\alpha(r,0)\). The series for \(T_\alpha'\) does not directly reveal its sign or the number of its zeros. The operator \(2-r\partial_r\) removes the
		\(r^2\) part of each radial term. We therefore define
		\[
		H_\alpha(r)=2T_\alpha(r)-rT_\alpha'(r).
		\]
		The identity
		\[
		2q_\mu(r)-rq_\mu'(r)=\frac{r^\mu}{\mu+2}
		\]
		also holds at \(\mu=2\). Applying it termwise to
		\eqref{eq:torsion-series}, we obtain, for \(0<r<1\),
		\begin{align}
			H_\alpha(r)
			&=\frac4\pi
			\sum_{k=0}^{\infty}
			\frac{(-1)^k}{2k+1}
			\frac{r^{(2k+1)\nu}}{(2k+1)\nu+2}\notag\\
			&=\frac4\pi\int_0^\infty
			e^{-2y}\arctan\bigl((re^{-y})^\nu\bigr)\,dy
			\label{eq:H-alpha-integral}\\
			&=\frac{4}{\pi r^2}
			\int_0^r s\arctan(s^\nu)\,ds.\notag
		\end{align}
		The interchange of summation and integration is justified
		by absolute convergence for \(r<1\). Differentiating the
		integral representation gives
		\[
		H_\alpha'(r)
		=\frac{4\nu}{\pi r}
		\int_0^\infty e^{-2y}
		\frac{(re^{-y})^\nu}
		{1+(re^{-y})^{2\nu}}\,dy>0.
		\]
		By the definition of \(H_\alpha\),
		\[
		\left(\frac{T_\alpha'(r)}r\right)'
		=-\frac{H_\alpha'(r)}{r^2}<0.
		\]
		Thus \(T_\alpha'\) has at most one zero in \((0,1)\).
		Since \(T_\alpha>0\) in \((0,1)\) and
		\(T_\alpha(0)=T_\alpha(1)=0\), it attains its maximum
		at an interior point. Its derivative vanishes there, so
		this point is the unique zero \(x_\alpha\) of
		\(T_\alpha'\).
		
		Because \(\frac{T_\alpha'(r)}{r}\) is strictly decreasing and
		vanishes at \(x_\alpha\),
		\[
		T_\alpha'(r)>0\quad(0<r<x_\alpha),
		\qquad
		T_\alpha'(r)<0\quad(x_\alpha<r<1).
		\]
		Moreover,
		\[
		H_\alpha'(x_\alpha)
		=-x_\alpha T_\alpha''(x_\alpha)>0,
		\]
		and therefore \(T_\alpha''(x_\alpha)<0\).

		Finally, \(v=-\partial_\theta u_\alpha\) is positive and harmonic
		in the upper half-sector and vanishes on the axis. Hopf's lemma
		on \(\theta=0\) gives
		\[
		\partial_\theta v(x_\alpha,0)>0,
		\qquad
		\partial_{\theta\theta}u_\alpha(x_\alpha,0)<0.
		\]
		Symmetry also gives
		\(\partial_{r\theta}u_\alpha(x_\alpha,0)=0\).
		At this critical point, the Hessian in the orthonormal polar frame is
		\[
		D^2u_\alpha(x_\alpha,0)=
		\begin{pmatrix}
			T_\alpha''(x_\alpha)&0\\
			0&x_\alpha^{-2}\partial_{\theta\theta}u_\alpha(x_\alpha,0)
		\end{pmatrix}.
		\]
		Both diagonal entries are negative, so the maximum is nondegenerate.
	\end{proof}
	
	\section{Monotonicity of the sectorial hot spots}
	\label{sec:sector-monotonicity}
	
	\subsection{Strict monotonicity of the torsion hot spot}
	
	We begin with the torsion hot spot, and show that it moves strictly toward the vertex as the sector opens.
	
	\begin{theorem}
		\label{thm:x-monotone}
		The function \(x_\alpha\) is strictly decreasing on
		\((0,2\pi)\).
	\end{theorem}
	
	\begin{proof}
		Fix \(0<\alpha_1<\alpha_2<2\pi\), and write
		\[
		\nu_i=\frac{\pi}{\alpha_i},
		\qquad
		H_i=H_{\alpha_i},
		\qquad
		x_i=x_{\alpha_i},
		\qquad i=1,2.
		\]
		Since \(T_\alpha(1)=0\), the definition of \(H_\alpha\) yields
		\begin{equation}\label{eq:T-from-H}
			T_\alpha(r)
			=r^2\int_r^1\frac{H_\alpha(s)}{s^3}\,ds.
		\end{equation}
		Differentiating gives
		\begin{equation}\label{eq:T-prime-from-H}
			\frac{T_\alpha'(r)}{r}
			=
			2\int_r^1\frac{H_\alpha(s)}{s^3}\,ds
			-\frac{H_\alpha(r)}{r^2}.
		\end{equation}
		
		At \(r=x_1\), the identity \(T_{\alpha_1}'(x_1)=0\) gives
		\[
		\frac{T_{\alpha_2}'(x_1)}{x_1}
		=
		2\int_{x_1}^1\frac{H_1(s)}{s^3}
		\left(
		\frac{H_2(s)}{H_1(s)}
		-\frac{H_2(x_1)}{H_1(x_1)}
		\right)ds.
		\]
		It therefore remains to show that $H_2/H_1$ is strictly decreasing.
		For $0<a<b$,
		\begin{equation}\label{eq:arctan-ratio}
			r\longmapsto\frac{\arctan(r^a)}{\arctan(r^b)}
			\quad\text{is strictly decreasing on }(0,1).
		\end{equation}
		Indeed,
		\[
		\frac{d}{dz}\frac{z}{(1+z^2)\arctan z}
		=\frac{(1-z^2)\arctan z-z}{(1+z^2)^2(\arctan z)^2}<0,
		\qquad 0<z<1.
		\]
		Since $r^b<r^a$ and $a<b$, this gives
		\[
		\frac{d}{dr}\log\frac{\arctan(r^a)}{\arctan(r^b)}
		=\frac1r\left[
		\frac{ar^a}{(1+r^{2a})\arctan(r^a)}
		-\frac{br^b}{(1+r^{2b})\arctan(r^b)}\right]<0.
		\]
		By \eqref{eq:H-alpha-integral} and $\nu_2<\nu_1$,
		\[
		\frac{H_2(r)}{H_1(r)}
		=\frac{\displaystyle\int_0^r s\arctan(s^{\nu_1})
			\frac{\arctan(s^{\nu_2})}{\arctan(s^{\nu_1})}\,ds}
		{\displaystyle\int_0^r s\arctan(s^{\nu_1})\,ds}
		>\frac{\arctan(r^{\nu_2})}{\arctan(r^{\nu_1})}.
		\]
		Consequently,
		\[
		\left(\frac{H_2}{H_1}\right)'(r)
		=\frac{r\arctan(r^{\nu_1})}
		{\displaystyle\int_0^r s\arctan(s^{\nu_1})\,ds}
		\left[\frac{\arctan(r^{\nu_2})}{\arctan(r^{\nu_1})}
		-\frac{H_2(r)}{H_1(r)}\right]<0.
		\]
		Thus $T'_{\alpha_2}(x_1)<0$. It follows that $x_2<x_1$.
	\end{proof}
	
	\subsection{Strict monotonicity of the ground-state hot spot}
	
	\begingroup
	We now turn to the ground-state hot spot in the sector.
	Let \(\nu=\pi/\alpha\). Up to a positive multiplicative constant,
	the first Dirichlet eigenfunction in \(S(\alpha)\) is
	\[
	\phi_\alpha(r,\theta)
	=J_\nu(j_{\nu,1}r)\cos(\nu\theta),
	\]
	where \(j_{\nu,1}\) is the first positive zero of \(J_\nu\).
	Since \(J_\nu(j_{\nu,1}r)>0\) for \(0<r<1\) and the angular
	factor has its unique maximum at \(\theta=0\), the hot spot
	lies on the symmetry axis. There, the radial critical points satisfy
	\[
	\frac{d}{dr}J_\nu(j_{\nu,1}r)
	=j_{\nu,1}J_\nu'(j_{\nu,1}r)=0.
	\]
	The interlacing $j'_{\nu,1}<j_{\nu,1}<j'_{\nu,2}$
	\cite[Ch.~XV]{watson} shows that the radial factor is strictly
	increasing before $\frac{j'_{\nu,1}}{j_{\nu,1}}$ and strictly decreasing
	afterwards. Therefore
	\begin{equation}\label{eq:y-alpha-bessel}
		y_\alpha=\frac{j'_{\nu,1}}{j_{\nu,1}},
		\qquad \nu=\frac{\pi}{\alpha},
	\end{equation}
	where \(j'_{\nu,1}\) is the first positive zero of \(J_\nu'\).
	\par\endgroup
	
	With the explicit Bessel formula for \(\phi_\alpha\), the monotonicity of the hot spot becomes a statement about how the ratio \(j'_{\nu,1}/j_{\nu,1}\) depends on the order.
	
	\begin{proposition}
		The function \(y_\alpha\) is strictly decreasing on \((0,2\pi)\).
	\end{proposition}
	
	\begin{proof}
		Fix $1/2<\nu_1<\nu_2$, and set
		\[
		\lambda_i=j_{\nu_i,1}^2,
		\qquad u_i(r)=J_{\nu_i}(j_{\nu_i,1}r),
		\qquad i=1,2.
		\]
		Then \(u_i>0\) on \((0,1)\), \(u_i(1)=0\), and
		\[
		(ru_i')'+\left(\lambda_i r-\frac{\nu_i^2}{r}\right)u_i=0.
		\]
		For
		\[
		\mathcal W(r)=r(u_1u_2'-u_2u_1'),
		\]
		these equations give
		\[
		\mathcal W'(r)
		=(\nu_2^2-\nu_1^2)\frac{1-cr^2}{r}u_1(r)u_2(r),
		\qquad c=\frac{\lambda_2-\lambda_1}{\nu_2^2-\nu_1^2},
		\]
		The asymptotics \(u_i=O(r^{\nu_i})\) and
		\(u_i'=O(r^{\nu_i-1})\) imply
		\[
		\lim_{r\downarrow0}\mathcal W(r)=0,
		\qquad \mathcal W(1)=0.
		\]
		Since the first eigenfunction is $J_\nu(j_{\nu,1}r)\cos(\nu\theta)$,
		the Rayleigh minimum is attained within the class
		$v(r)\cos(\nu\theta)$. Integrating in $\theta$ gives
		\[
		j_{\nu,1}^2=
		\inf_{v\ne0}
		\frac{\displaystyle\int_0^1
			\left(r|v'|^2+\frac{\nu^2}{r}|v|^2\right)\,dr}
		{\displaystyle\int_0^1r|v|^2\,dr},
		\]
		where \(v\) is locally absolutely continuous, \(v(1)=0\),
		and the numerator is finite. 
		The behavior $u_2=O(r^{\nu_2})$ and $u_2'=O(r^{\nu_2-1})$ at zero makes $u_2$ admissible for the $\nu_1$-problem.
		Taking \(u_2\) as a test
		function for the \(\nu_1\)-problem yields
		\[
		\begin{aligned}
			\lambda_2
			&\ge\lambda_1+(\nu_2^2-\nu_1^2)
			\frac{\displaystyle\int_0^1\frac{u_2(r)^2}{r}\,dr}
			{\displaystyle\int_0^1r u_2(r)^2\,dr}\\
			&>\lambda_1+\nu_2^2-\nu_1^2.
		\end{aligned}
		\]
		Hence \(c>1\). Since \(\mathcal W'\) changes sign once,
		from positive to negative, the endpoint conditions give
		\(\mathcal W>0\) on \((0,1)\). Consequently,
		\[
		\frac{\mathcal W(r)}{r u_1(r)u_2(r)}=\frac{u_2'(r)}{u_2(r)}-\frac{u_1'(r)}{u_1(r)}>0.
		\]
		Let
		\[
		y_i=\frac{j'_{\nu_i,1}}{j_{\nu_i,1}},
		\qquad i=1,2.
		\]
		At \(r=y_1\), this inequality gives \(u_2'(y_1)>0\),
		so \(y_1<y_2\). Thus
		$
		\nu\longmapsto\frac{j'_{\nu,1}}{j_{\nu,1}}
		$
		is strictly increasing. Since \(\nu=\pi/\alpha\),
		the assertion follows.
	\end{proof}
	
	\medskip
	
	Although both hot spots move toward the vertex as the opening angle
	increases, their relative order may change. We compare their
	small-angle asymptotics and half-disk values, and determine their
	limits as \(\alpha\uparrow2\pi\).
	
	\section{Comparison of the two hot spots over some special sectors}
	\label{sec:sector-comparison}
	
	\subsection{Continuity in the opening angle}
	
	The integral formulas \eqref{eq:H-alpha-integral} and
	\eqref{eq:T-prime-from-H} give joint continuity of
	\(T_\alpha'(r)\) for \(0<\alpha<2\pi\), \(0<r<1\).
	The strict sign change at its unique zero implies continuity of
	\(x_\alpha\): for \(r_-<x_{\alpha_0}<r_+\), the signs
	\(T_{\alpha_0}'(r_-)>0>T_{\alpha_0}'(r_+)\) persist for
	\(\alpha\) near \(\alpha_0\), trapping \(x_\alpha\) between
	\(r_-\) and \(r_+\).
	Continuity of \(y_\alpha\) follows from smooth parameter dependence
	and simplicity of the first Bessel zeros \cite[\S10.21(i)]{dlmf}.
	In particular, \(x_\alpha-y_\alpha\) is continuous on \((0,2\pi)\).
	
	\subsection{Small-angle asymptotics}
	
	Set
	\[
	\beta=\frac{2\alpha}{\pi},
	\qquad
	\nu=\frac{\pi}{\alpha}=\frac{2}{\beta}.
	\]
	
	When the sector becomes sharp, both hot spots are pushed toward the arc, but at different rates; the comparison of the two rates decides which of them lies farther out.
	
	\begin{proposition}
		\label{prop:small-angle-sector}
		As \(\beta\downarrow0\),
		\begin{equation}
			1-x_\alpha=
			\frac{\beta}{2}\log\frac1\beta
			-\frac{\beta}{2}\log\frac{\pi^3}{32}
			+O\!\left(\beta^2\log^2\frac1\beta\right),
			\label{eq:torsion-small-angle}
		\end{equation}
		and
		\begin{equation}
			1-y_\alpha=
			\frac{|a_1|-|a'_1|}{2}\beta^{2/3}
			+O(\beta^{4/3}),
			\label{eq:eigen-small-angle}
		\end{equation}
		where \(a_1<0\) and \(a'_1<0\) are the first zeros of
		\(\operatorname{Ai}\) and \(\operatorname{Ai}'\), respectively.
		Consequently,
		\begin{equation}
			x_\alpha-y_\alpha=
			\frac{|a_1|-|a'_1|}{2}\beta^{2/3}
			-\frac{\beta}{2}\log\frac1\beta
			+O(\beta),
			\label{eq:sector-difference-small-angle}
		\end{equation}
		and \(x_\alpha>y_\alpha\) for all sufficiently small
		\(\alpha>0\).
	\end{proposition}
	
	\begin{proof}
		We first consider the torsion hot spot. Put
		\[
		t=x_\alpha^{2/\beta},
		\qquad\text{so that}\qquad
		x_\alpha=t^{\beta/2}.
		\]
		For \(0<\beta<1\), substituting
		\(\mu_k=\frac{2(2k+1)}{\beta}\) into
		\eqref{eq:x-alpha-equation} gives
		\begin{equation}\label{eq:Q-eq}
			Q_\beta(t)=C_\beta,
		\end{equation}
		where
		\begin{equation}\label{eq:Qbeta}
			Q_\beta(t)
			=t^{-\beta}
			\sum_{k=0}^{\infty}
			\frac{(-1)^k t^{2k+1}}
			{(2k+1)^2-\beta^2},
		\end{equation}
		and
		\begin{equation}\label{eq:Cbeta}
			C_\beta
			=\beta\sum_{k=0}^{\infty}
			\frac{(-1)^k}
			{(2k+1)((2k+1)^2-\beta^2)}.
		\end{equation}
		As \(\beta\downarrow0\),
		\[
		C_\beta=\frac{\pi^3}{32}\beta+O(\beta^3),
		\qquad
		\sum_{k=0}^{\infty}\frac{(-1)^k}{(2k+1)^3}
		=\frac{\pi^3}{32}.
		\]
		For \(0<\beta\le1/2\), the alternating-series bounds give
		\[
		Q_\beta(t)\ge\frac89t^{1-\beta},
		\qquad C_\beta\le\frac43\beta.
		\]
		Hence \(t\le t^{1-\beta}\le\frac{3\beta}{2}\). Expanding
		\eqref{eq:Q-eq} with \(t=O(\beta)\), we obtain
		\[
		t^{1-\beta}=\frac{\pi^3}{32}\beta\bigl(1+O(\beta^2)\bigr),
		\]
		and therefore
		\[
		\log x_\alpha
		=\frac\beta{2(1-\beta)}
		\left(\log\beta+\log\frac{\pi^3}{32}+O(\beta^2)\right).
		\]
		Exponentiation gives \eqref{eq:torsion-small-angle}.
		
		For the ground state, the large-order expansions
		\cite[Eqs.~(10.21.32), (10.21.36)]{dlmf} give
		\[
		\begin{aligned}
			j_{\nu,1}&=\nu+2^{-1/3}|a_1|\nu^{1/3}+O(\nu^{-1/3}),\\
			j'_{\nu,1}&=\nu+2^{-1/3}|a'_1|\nu^{1/3}+O(\nu^{-1/3}).
		\end{aligned}
		\]
		Together with \eqref{eq:y-alpha-bessel}, these formulas imply
		\[
		y_\alpha
		=1-2^{-1/3}(|a_1|-|a'_1|)\nu^{-2/3}
		+O(\nu^{-4/3}).
		\]
		Substituting \(\nu=2/\beta\) proves
		\eqref{eq:eigen-small-angle}. Finally,
		\eqref{eq:sector-difference-small-angle} follows by subtraction.
		Since \(|a_1|>|a'_1|\) and
		\[
		\beta^{2/3}\gg\beta\log\frac1\beta
		\qquad(\beta\downarrow0),
		\]
		its leading term is positive. Thus \(x_\alpha>y_\alpha\) for
		all sufficiently small \(\alpha>0\).
	\end{proof}
	
	\subsection{The half-disk comparison}
	
	At \(\alpha=\pi\), Corollary~\ref{cor:half-disk} gives
	\[
	T_\pi'(r)=-r+\frac{1+r^{-2}}\pi
	+\frac{r-r^{-3}}\pi\arctan r.
	\]
	Its unique zero and the Bessel expression \eqref{eq:y-alpha-bessel}
	yield the numerical values
	\[
	x_\pi\approx0.4802197170,
	\qquad y_\pi=\frac{j'_{1,1}}{j_{1,1}}\approx0.4805128044.
	\]
	The strict inequality \(x_\pi<y_\pi\) is established in
	\cite[Proposition~2.1]{bensonlaugesen}.
	
	The difference has opposite signs at the two ends of the range of angles, so continuity forces a coincidence.
	
	\begin{corollary}
		There exists \(\alpha_0\in(0,\pi)\) such that
		\[
		x_{\alpha_0}=y_{\alpha_0}.
		\]
	\end{corollary}
	
	\begin{proof}
		By Proposition~\ref{prop:small-angle-sector}, \(x_\alpha-y_\alpha>0\)
		for all sufficiently small \(\alpha>0\), whereas \(x_\pi-y_\pi<0\).
		Continuity and the intermediate value theorem give the assertion.
	\end{proof}

	\subsection{The slit-disk limit as the angle approaches \(2\pi\)}
	
	The increasing union of the sectors is the slit disk
	\[
	\bigcup_{0<\alpha<2\pi}S(\alpha)
	=B_1(0)\setminus\{(x,0):-1<x\le0\}.
	\]
	\begingroup
	
	The limiting Dirichlet problem retains both faces of the slit.
	Strict monotonicity ensures that both coordinates have limits, and the two limits are
	characterised by scalar equations, one for each hot spot. Remark~\ref{rem:slit-arithmetic}
	presents arithmetic consequences of those equations.
	\par\endgroup
	
	\begin{proposition}
		\label{prop:slit-limits}
		Let \(z_*\) be the unique root of \(\tan z=2z\) in
		\((0,\pi/2)\), and let \(t_*\) be the unique root in \((0,1)\) of
		\[
		3(1-t^8)\arctan t=t(3t^6-t^4-t^2+3).
		\]
		Set \(x_*=t_*^2\) and \(y_*=\frac{z_*}{\pi}\). Then
		\[
		\lim_{\alpha\uparrow2\pi}x_\alpha=x_*,
		\qquad
		\lim_{\alpha\uparrow2\pi}y_\alpha=y_*,
		\]
		and
		\[
		\{x_\alpha:0<\alpha<2\pi\}=(x_*,1),
		\qquad
		\{y_\alpha:0<\alpha<2\pi\}=(y_*,1).
		\]
	\end{proposition}
	
	\begin{proof}
		For the ground state, \(\nu\downarrow1/2\), and
		\[
		J_{\frac{1}{2}}(z)=\sqrt{\frac{2}{\pi z}}\sin z,
		\qquad j_{\frac{1}{2},1}=\pi
		\]
		\cite[Eq.~(10.16.1)]{dlmf}. Continuity of the first zeros gives
		\[
		\lim_{\alpha\uparrow2\pi}y_\alpha
		=\frac{j'_{\frac{1}{2},1}}{\pi}=:y_*.
		\]
		The limiting radial factor is proportional to
		\(r^{-1/2}\sin(\pi r)\). Its derivative vanishes precisely when
		\[
		2z\cos z-\sin z=0,\qquad z=\pi r.
		\]
		Equivalently, \(z_*\) is the unique root of \(\tan z=2z\) in
		\((0,\pi/2)\), and \(y_*=\frac{z_*}{\pi}\).
		
		For the torsion function, extend the positive integral formulas to
		\(\nu=1/2\):
		\[
		H_*(r)=\frac4{\pi r^2}\int_0^r s\arctan\sqrt{s}\,ds,
		\qquad
		T_*(r)=r^2\int_r^1\frac{H_*(s)}{s^3}\,ds.
		\]
		Since
		\[
		0<H_*(r)\le\frac8{5\pi}\sqrt r,
		\qquad 0<T_*(r)\le\frac{16}{15\pi}\sqrt r,
		\]
		\(T_*(0)=T_*(1)=0\). Moreover,
		\(\left(\frac{T_*'}{r}\right)'=-\frac{H_*'}{r^2}<0\), so \(T_*\) has a unique critical
		point \(x_*\). Dominated convergence gives
		\(T_\alpha'\to T_*'\) locally uniformly on \((0,1)\);
		the strict sign change at \(x_*\) then yields
		\[
		\lim_{\alpha\uparrow2\pi}x_\alpha=x_*.
		\]
		
		At \(\nu=1/2\), summing the axial series by partial fractions gives,
		with \(t=\sqrt r\),
		\[
		T_*(t^2)=-\frac1{2\pi}\left[
		(t^4+t^{-4}-2)\arctan t+t^3-\frac t3-t^{-3}+\frac1{3t}
		\right].
		\]
		Differentiation shows that \(x_*=t_*^2\), where \(t_*\in(0,1)\)
		is the unique root of
		\[
		3(1-t^8)\arctan t=t(3t^6-t^4-t^2+3).
		\]
		Finally, the small-angle limits, continuity, and strict monotonicity
		show that the ranges of \(x_\alpha\) and \(y_\alpha\) are
		\((x_*,1)\) and \((y_*,1)\), respectively.
	\end{proof}
	
	Solving the two scalar equations numerically gives the illustrative values
	\[
	x_*\approx0.3695248800,
	\qquad y_*\approx0.3710096482.
	\]
	
	For radius \(R\), both coordinates scale by \(R\). The nonzero
	limiting radii distinguish the slit disk from the full disk, where
	both maxima lie at the center.
	
	The two equations are transcendental, but elementary enough to yield arithmetic information about the limits.
	
	\begin{remark}
		\label{rem:slit-arithmetic}
		The numbers $x_*$ and $\pi y_*$ are transcendental, and $y_*$ is
		irrational.
		
		Indeed, suppose that $x_*$ were algebraic. Then $t_*=\sqrt{x_*}$
		would also be algebraic. Its defining equation gives
		\[
		a_*:=\arctan t_*
		=\frac{t_*(3t_*^6-t_*^4-t_*^2+3)}{3(1-t_*^8)},
		\]
		so $a_*>0$ would be algebraic. However,
		\[
		e^{2ia_*}=\frac{1+it_*}{1-it_*}
		\]
		would then be algebraic, contradicting the Hermite--Lindemann
		theorem. Thus $x_*$ is transcendental.
		
		Similarly, if $z_*=\pi y_*$ were algebraic, the identity
		$\tan z_*=2z_*$ would imply
		\[
		e^{2iz_*}=\frac{1+2iz_*}{1-2iz_*}
		\]
		is algebraic, again contradicting the same theorem. Hence
		$\pi y_*$ is transcendental.
		
		Finally, if $y_*$ were rational, then $e^{2\pi i y_*}$ would
		be a root of unity, and consequently
		\[
		\tan(\pi y_*)
		=\frac{e^{2\pi i y_*}-1}{i(e^{2\pi i y_*}+1)}
		\]
		would be algebraic. This contradicts
		$\tan(\pi y_*)=2\pi y_*$, since $y_*>0$ is rational and $\pi$
		is transcendental. Thus $y_*$ is irrational.
		The transcendence of $\pi y_*$ alone does not establish
		the transcendence of $y_*$.
	\end{remark}
	
	\subsection{Numerical comparison on the full angle range}
	
	Let
	\[
	d(\alpha)=x_\alpha-y_\alpha.
	\]
	The small-angle asymptotics show that
	\(d(\alpha)>0\) for all sufficiently small \(\alpha\) and that
	\(d(\alpha)\to0\) as \(\alpha\downarrow0\). The separate
	monotonicity of \(x_\alpha\) and \(y_\alpha\) does not determine
	the behavior of \(d\) on the whole interval
	\[
	0<\alpha<2\pi.
	\]
	
	We evaluated both coordinates on a fine grid of angles. What follows records only what the computation indicates.
	
	\begin{numericalobservation}
		The computation suggests a local maximum of $|x_\alpha-y_\alpha|$ near
		\[
		\begin{aligned}
			\alpha_*&\approx0.03287946\quad(1.883854^\circ),\\
			x_{\alpha_*}&\approx0.95918824,\qquad
			y_{\alpha_*}\approx0.95206383,\\
			|x_{\alpha_*}-y_{\alpha_*}|&\approx0.00712442.
		\end{aligned}
		\]
		The computation also indicates a sign change of $x_\alpha-y_\alpha$
		near $\alpha\approx0.56201409$ and a smaller local maximum of
		$|x_\alpha-y_\alpha|$ near $\alpha\approx1.040023$, with value
		approximately $0.00060947$. These observations suggest that
		$|x_\alpha-y_\alpha|$ is not unimodal; they neither establish global
		extremality nor exclude additional critical points.
	\end{numericalobservation}
	
	These observations can in principle be upgraded to a proof.
	
	\begin{remark}
		A rigorous verification would require certified bounds for the coordinates
		in \eqref{eq:x-alpha-equation} and \eqref{eq:y-alpha-bessel}.
		Writing $F(x,\alpha)=T'_\alpha(x)$, implicit differentiation gives
		\[
		x'_\alpha=-\frac{F_\alpha(x_\alpha,\alpha)}{F_x(x_\alpha,\alpha)},
		\]
		since $F_x(x_\alpha,\alpha)=T''_\alpha(x_\alpha)<0$.
		Together with parameter derivatives of the Bessel zeros, this permits
		interval estimates for $d'(\alpha)$. A global verification would need
		to isolate its zeros and compare the corresponding values of $|d|$,
		with separate control near $0$ and $2\pi$. The small-angle asymptotics
		establish $d(\alpha)>0$ near zero but do not locate the observed peak.
	\end{remark}
	
	Normalising by area rather than by the diameter destroys the bound.
	
	\begin{remark}
		In Beck's construction \cite[Proposition~2.4]{beck2020} the domains
		have area comparable to \(N\), while the separation is bounded below by
		a positive multiple of \(N\). After normalising by area the ratio
		therefore grows like \(N^{1/2}\), and no uniform bound can hold. Any
		estimate for the distance between the two hot spots must thus involve
		geometric restrictions, such as diameter control or membership in a
		one-parameter family such as sectors; area alone cannot bound it, even
		among convex planar domains.
	\end{remark}
	
	\section{The sharp diameter bound in symmetric convex domains}
	\label{sec:convex-diameter}
	
	For a bounded convex planar domain \(\Omega\) with a line of symmetry, let \(P_\Omega\) and \(Q_\Omega\) denote the torsion hot spot and the ground-state hot spot, respectively.  This section proves Theorem~\ref{thm:main}.  We first sharpen the upper bound by retaining the projection length and the maximal-width plateau.  We then construct a family of thin convex domains that proves sharpness.
	The upper bound has two ingredients: the heart confines both hot spots to an interval whose length is at
	most half the projection length, and an endpoint-reflection argument removes equality. For the lower bound
	we build thin convex domains whose relative width differs from a constant only
	by a tent-shaped bump of size \(\varepsilon^{2+\sigma}\) peaking near one end. On the axis the
	torsion function of a rectangle is almost flat, so a variation of order \(\varepsilon^{4+\sigma}\) is enough to relocate the maximum. For the first eigenfunction the transverse
	eigenvalue varies only by \(O(\varepsilon^{\sigma})\), which is negligible against the
	order-one longitudinal energy; the longitudinal profile therefore still approaches
	\(\sin(\pi x)\), whose peak is nondegenerate, and its maximum stays near the midpoint.
	The two maxima therefore separate at a scale of one half of the diameter.
	
	\begingroup
	
	We recall only the reflection geometry needed for the upper bound.
	For a convex body \(K\), a unit vector \(e\), and a real number
	\(\lambda\), let
	\[
	K_{\lambda,e}^+=\{z\in K:z\cdot e\ge\lambda\},
	\qquad
	T_{\lambda,e}z=z-2(z\cdot e-\lambda)e.
	\]
	The maximal folding value and the heart are
	\[
	\begin{aligned}
		R_K(e)
		&=\min\{\lambda:T_{\lambda,e}(K_{\lambda,e}^+)\subset K\},\\
		\mathcal H(K)
		&=\{z\in K:z\cdot e\le R_K(e)\text{ for every unit }e\}.
	\end{aligned}
	\]
	The standard reflection localization gives
	\[
	P_\Omega,Q_\Omega\in\mathcal H(\overline\Omega);
	\]
	see \cite[Section~3]{brascomagnaninisalani2011} and
	\cite[Section~2]{brascomagnanini2013}, building on
	\cite{gidasnirenberg,berestyckinirenberg}.
	For the two linear equations considered here, comparison on a proper
	cap \(D\) is immediate: it is harmonic comparison for torsion, while
	\(\lambda_1(D)>\lambda_1(\Omega)\) gives the maximum principle for
	\(-\Delta-\lambda_1(\Omega)\) in the eigenfunction case.
	These comparisons require no smoothness of the outer convex boundary;
	the Hopf argument below takes place on an interior straight line.
	\par\endgroup
	
	The heart confines both hot spots to a common interval. The next proposition estimates its
	length and excludes equality in the half-projection-length bound.
	
	\begin{proposition}\label{prop:upperbound}
		Let \(\Omega\subset\mathbb R^2\) be a bounded convex planar domain symmetric about a line \(\ell\).  Let
		\[
		L_\ell(\Omega)=
		|\operatorname{proj}_{\ell}\Omega|
		\]
		and let \(m_\ell(\Omega)\) be the length of the interval on which the
		orthogonal half-width of \(\Omega\) is maximal.  Then
		\[
		|P_\Omega-Q_\Omega|
		\le \frac{L_\ell(\Omega)-m_\ell(\Omega)}2.
		\]
		In addition,
		\[
		|P_\Omega-Q_\Omega|
		<\frac{L_\ell(\Omega)}2
		\le\frac{\operatorname{diam}(\Omega)}2.
		\]
	\end{proposition}
	
	\begin{proof}
		\begingroup
		After a rigid motion, write
		\[
		\Omega=\{(x,y):0<x<L,\ |y|<h(x)\},
		\qquad \operatorname*{arg\,max}_{[0,L]}h=[a,b],
		\]
		where \(h\) is concave and its endpoint values are understood by
		continuous extension.
		Reflection symmetry gives
		\(\mathcal H(\overline\Omega)\subset\{y=0\}\).
		Since \(h\) is nonincreasing on \([a,L]\), for
		\(\mu\ge(L+a)/2\) and \(\mu<x<L\) we have
		\[
		a\le 2\mu-x<x,
		\qquad h(2\mu-x)\ge h(x).
		\]
		Thus every such right cap reflects into \(\Omega\).
		Likewise, \(h\) is nondecreasing on \([0,b]\), so every left cap
		cut at \(\mu\le b/2\) reflects into \(\Omega\).
		The definition of the heart therefore gives
		\begin{equation}\label{eq:heart-window}
			\mathcal H(\overline\Omega)
			\subset
			\left[\frac b2,\frac{L+a}{2}\right]\times\{0\}.
		\end{equation}
		Consequently,
		\[
		|P_\Omega-Q_\Omega|
		\le\operatorname{diam}\mathcal H(\overline\Omega)
		\le\frac{L-(b-a)}2
		=\frac{L_\ell(\Omega)-m_\ell(\Omega)}2.
		\]
		
		For strictness with respect to \(L/2\), only \(a=b=:c\) needs
		consideration. If \(c=0\) or \(c=L\), one endpoint of the common
		interval is on \(\partial\Omega\) and cannot be a hot spot.
		If \(0<c<L\), reflect the left cap
		\(D=\Omega\cap\{x<c/2\}\) across \(x=c/2\).
		For either \(v=u\) or \(v=\phi_1\), the reflected difference
		\[
		w(x,y)=v(c-x,y)-v(x,y)
		\]
		is nonnegative in \(D\). It is not identically zero, since its
		continuous boundary value at \((0,0)\) is \(v(c,0)>0\).
		The strong maximum principle gives \(w>0\) in \(D\), and Hopf's lemma
		at the interior point \((c/2,0)\) of the reflection line yields
		\(v_x(c/2,0)>0\).
		Thus neither hot spot is the left endpoint of the common interval,
		and equality with its length \(L/2\) is impossible.
		Finally, \(L=L_\ell(\Omega)\le\operatorname{diam}(\Omega)\).\qedhere
		\par\endgroup
	\end{proof}
	
	For the sharpness construction we present two elementary estimates.  The first says that, in a sufficiently long rectangle, the torsion value at the center is exponentially close to the value for the infinite strip.
	
	\begin{lemma}\label{lem:long-rectangle}
		Let
		\[
		R_{L,a}=(-\frac{L}{2},\frac{L}{2})\times(-a,a),
		\]
		and let \(v\) be the torsion function in \(R_{L,a}\):
		\[
		-\Delta v=1\quad\text{in }R_{L,a},
		\qquad
		v=0\quad\text{on }\partial R_{L,a}.
		\]
		Then there exist absolute constants \(C,c>0\) such that
		\[
		v(0,0)\ge
		\frac{a^2}{2}\left(1-Ce^{-cL/a}\right).
		\]
	\end{lemma}
	
	\begin{proof}
		The torsion function in the infinite strip \(\mathbb R\times(-a,a)\) is
		\[
		U(y)=\frac{a^2-y^2}{2}.
		\]
		Set $H=U-v$. Then \(H\) is harmonic in the rectangle, vanishes on the top and bottom sides \(y=\pm a\), and equals \(U(y)\) on the vertical sides \(x=\pm L/2\).
		
		Expand \(U\) in the Dirichlet eigenfunctions on \((-a,a)\):
		\[
		U(y)=\sum_{m=0}^{\infty} b_m
		\cos(\lambda_m y),
		\]
		where
		\[
		b_m=\frac{(-1)^m16a^2}{\pi^3(2m+1)^3},
		\qquad
		\lambda_m=\frac{(2m+1)\pi}{2a}.
		\]
		For each vertical mode, the horizontal factor solves
		$h_m''-\lambda_m^2h_m=0$ with $h_m(\pm L/2)=1$. Thus
		\[
		H(x,y)=
		\sum_{m=0}^{\infty}
		b_m
		\frac{\cosh(\lambda_m x)}{\cosh(\lambda_m L/2)}
		\cos(\lambda_m y).
		\]
		Since $\sum_m|b_m|\le C_0a^2$, this series converges uniformly on the
		closed rectangle and has the prescribed boundary values. Its sum is
		harmonic in the interior, so uniqueness of the Dirichlet problem
		identifies it with $H(x,y)$.
		
		Since
		\[
		\frac1{\cosh(\lambda_m L/2)}
		\le 2e^{-\lambda_m L/2}
		\le 2e^{-\pi L/(4a)},
		\]
		we obtain
		\[
		v(0,0)=U(0)-H(0,0)
		\ge \frac{a^2}{2}-C_1a^2e^{-\pi L/(4a)}.
		\]
		After adjusting the constants, this gives the claim.
	\end{proof}
	
	\begin{remark}
		The exponential error also has a diffusion interpretation.
		The correction $H$ comes from paths that reach
		an end before being absorbed by the horizontal sides.
		To reach a distant end, a particle must cross many successive
		portions of the strip, each of length comparable to its width,
		without touching either side.
		Since the width is fixed, each crossing carries a uniformly
		positive risk of absorption, even for particles that have survived
		all previous crossings.
		The probability of reaching a distant end therefore decreases
		exponentially with the number of crossings, explaining the
		exponential dependence on $L/a$.
	\end{remark}
	
	The second estimate identifies the limiting position of the ground-state hot spot in a thin strip whose width differs from a constant only at higher order.
	
	\begin{lemma}\label{lem:near-rectangle-eigen}
		Let
		\[
		\Omega_\varepsilon=
		\{(x,y):0<x<1,\ |y|<\varepsilon f_\varepsilon(x)\},
		\]
		where \(f_\varepsilon>0\) and \(\Omega_\varepsilon\) is convex.
		Assume only that
		\[
		\delta_\varepsilon
		:=\|f_\varepsilon-1\|_{L^\infty(0,1)}
		=o(\varepsilon^2).
		\]
		Let \(\phi_\varepsilon\) be the positive first Dirichlet eigenfunction in
		\(\Omega_\varepsilon\), with eigenvalue \(\lambda_\varepsilon\). If
		\[
		Q_{\Omega_\varepsilon}=(q_\varepsilon,0),
		\]
		then
		\[
		q_\varepsilon\longrightarrow\frac12.
		\]
	\end{lemma}
	
	\begin{proof}
		Apply the anisotropic scaling \(y=\varepsilon Y\), without
		flattening the variable boundary, and set
		\[
		D_\varepsilon
		=
		\{(x,Y):0<x<1,\ |Y|<f_\varepsilon(x)\}.
		\]
		After multiplication by a positive constant, the rescaled eigenfunction
		\[
		v_\varepsilon(x,Y)=\phi_\varepsilon(x,\varepsilon Y)
		\]
		may be normalized by
		\[
		\int_{D_\varepsilon}v_\varepsilon^2\,dx\,dY=1.
		\]
		Put ${h_\varepsilon}^\pm=1\pm\delta_\varepsilon$. Then
		\[
		R_\varepsilon^-
		:=(0,1)\times(-{h_\varepsilon}^-,{h_\varepsilon}^-)
		\subset D_\varepsilon
		\subset
		R_\varepsilon^+
		:=(0,1)\times(-{h_\varepsilon}^+,{h_\varepsilon}^+).
		\]
		Domain monotonicity and the explicit first eigenvalue of a rectangle give
		\[
		\pi^2+\frac{\pi^2}{4\varepsilon^2({h_\varepsilon}^+)^2}
		\le \lambda_\varepsilon
		\le
		\pi^2+\frac{\pi^2}{4\varepsilon^2({h_\varepsilon}^-)^2}.
		\]
		Since \(\delta_\varepsilon=o(\varepsilon^2)\),
		\begin{equation}\label{eq:thin-eigenvalue-sandwich}
			\lambda_\varepsilon
			=
			\pi^2+\frac{\pi^2}{4\varepsilon^2({h_\varepsilon}^+)^2}+o(1).
		\end{equation}
		
		Extend \(v_\varepsilon\) by zero from \(D_\varepsilon\) to
		\(R_\varepsilon^+\). This extension belongs to
		\(H_0^1(R_\varepsilon^+)\), and its Rayleigh identity is
		\[
		\lambda_\varepsilon
		=
		\int_{R_\varepsilon^+}
		\left(
		|\partial_xv_\varepsilon|^2
		+\varepsilon^{-2}|\partial_Yv_\varepsilon|^2
		\right)dx\,dY.
		\]
		The one-dimensional Poincar\'e inequalities in \(x\) and \(Y\) imply
		\[
		\int_{R_\varepsilon^+}|\partial_xv_\varepsilon|^2\,dx\,dY
		\ge\pi^2,
		\qquad
		\int_{R_\varepsilon^+}|\partial_Yv_\varepsilon|^2\,dx\,dY
		\ge\frac{\pi^2}{4({h_\varepsilon}^+)^2}.
		\]
		Together with \eqref{eq:thin-eigenvalue-sandwich}, this yields
		\begin{equation}\label{eq:thin-energy-excess}
			\int_{R_\varepsilon^+}|\partial_xv_\varepsilon|^2\,dx\,dY
			=\pi^2+o(1),
			\qquad
			\int_{R_\varepsilon^+}|\partial_Yv_\varepsilon|^2\,dx\,dY
			=\frac{\pi^2}{4({h_\varepsilon}^+)^2}+o(\varepsilon^2).
		\end{equation}
		
		Let
		\[
		\chi_\varepsilon(Y)
		=({h_\varepsilon}^+)^{-1/2}
		\cos\frac{\pi Y}{2{h_\varepsilon}^+},
		\qquad |Y|<{h_\varepsilon}^+,
		\]
		be the normalized first Dirichlet eigenfunction of the vertical interval.
		Set
		\[
		g_\varepsilon(x)=\int_{-h_\varepsilon^+}^{h_\varepsilon^+}
		v_\varepsilon(x,Y)\chi_\varepsilon(Y)\,dY
		\]
		and decompose
		\[
		v_\varepsilon(x,Y)
		=g_\varepsilon(x)\chi_\varepsilon(Y)+r_\varepsilon(x,Y),
		\qquad
		\int_{-{h_\varepsilon}^+}^{{h_\varepsilon}^+}
		r_\varepsilon(x,Y)\chi_\varepsilon(Y)\,dY=0.
		\]
		Orthogonality and normalization give
		\begin{align}
			\label{eq:normalyunnan}
			\|g_\varepsilon\|_{L^2(0,1)}^2
			+\|r_\varepsilon\|_{L^2(R_\varepsilon^+)}^2=1
		\end{align}
		\begin{align}
			\label{eq:orthyunnan}
			\int_{R_\varepsilon^+}|\partial_Yv_\varepsilon|^2
			=\frac{\pi^2}{4(h_\varepsilon^+)^2}\|g_\varepsilon\|_{L^2(0,1)}^2
			+\|\partial_Yr_\varepsilon\|_{L^2(R_\varepsilon^+)}^2.
		\end{align}
		Since $r_\varepsilon(x,\cdot)$ is
		orthogonal to $\chi_\varepsilon$, its vertical Rayleigh quotient
		is at least the second vertical eigenvalue 
		$\pi^2/(h_\varepsilon^+)^2$. Hence by \eqref{eq:normalyunnan} and \eqref{eq:orthyunnan}, we have
		\[
		\int_{R_\varepsilon^+}|\partial_Yv_\varepsilon|^2
		\ge
		\frac{\pi^2}{4({h_\varepsilon}^+)^2}
		+\frac{3\pi^2}{4({h_\varepsilon}^+)^2}
		\int_{R_\varepsilon^+}r_\varepsilon^2.
		\]
		By \eqref{eq:thin-energy-excess}, this first gives
		$\|r_\varepsilon\|_{L^2(R_\varepsilon^+)}^2=o(\varepsilon^2)$.
		The preceding identity and normalization then give
		\[
		\|\partial_Yr_\varepsilon\|_{L^2(R_\varepsilon^+)}^2
		=\frac{\pi^2}{4(h_\varepsilon^+)^2}
		\|r_\varepsilon\|_{L^2(R_\varepsilon^+)}^2
		+o(\varepsilon^2)=o(\varepsilon^2).
		\]
		Thus
		\begin{equation}\label{eq:thin-remainder}
			\|r_\varepsilon\|_{L^2(R_\varepsilon^+)}
			=o(\varepsilon),
			\qquad
			\|\partial_Yr_\varepsilon\|_{L^2(R_\varepsilon^+)}
			=o(\varepsilon).
		\end{equation}
		Moreover,
		\[
		\|g_\varepsilon\|_{L^2(0,1)}\to1,
		\qquad
		\int_0^1|g_\varepsilon'|^2\,dx
		\le
		\int_{R_\varepsilon^+}|\partial_xv_\varepsilon|^2\,dx\,dY
		=\pi^2+o(1).
		\]
		Since $g_\varepsilon\in H_0^1(0,1)$ and $g_\varepsilon\ge0$,
		compactness shows that every sequence $\varepsilon\to0$ admits a
		subsequence converging weakly in $H_0^1(0,1)$ and strongly in $L^2(0,1)$
		to some $g\ge0$ with $\|g\|_{L^2(0,1)}=1$ and
		$\|g'\|_{L^2(0,1)}^2\le\pi^2$.
		Equality in the one-dimensional Poincar\'e inequality forces
		$g=\sqrt2\sin(\pi x)$. Moreover, the same inequality and the preceding
		bounds give $\|g_\varepsilon'\|_{L^2(0,1)}^2\to\pi^2$, so the weak
		convergence is strong. Since the limit is unique,
		\begin{equation}\label{eq:thin-first-mode}
			g_\varepsilon\longrightarrow\sqrt2\sin(\pi x)
			\quad\text{in }H_0^1(0,1),
		\end{equation}
		and therefore uniformly on \([0,1]\).
		
		Since $r_\varepsilon(x,-h_\varepsilon^+)=0$, for almost every $x$,
		\[
		r_\varepsilon(x,0)
		=\int_{-h_\varepsilon^+}^0\partial_Yr_\varepsilon(x,Y)\,dY.
		\]
		Cauchy--Schwarz and \eqref{eq:thin-remainder} therefore give
		\[
		\|r_\varepsilon(\cdot,0)\|_{L^2(0,1)}^2
		\le h_\varepsilon^+
		\int_0^1\int_{-h_\varepsilon^+}^0
		|\partial_Yr_\varepsilon|^2\,dY\,dx
		\le h_\varepsilon^+
		\|\partial_Yr_\varepsilon\|_{L^2(R_\varepsilon^+)}^2
		=o(\varepsilon^2).
		\]
		Since $\chi_\varepsilon(0)\to1$, \eqref{eq:thin-first-mode} yields
		\begin{equation}\label{eq:thin-axis-convergence}
			v_\varepsilon(\cdot,0)\longrightarrow\sqrt2\sin(\pi x)
			\quad\text{in }L^2(0,1).
		\end{equation}
		Finally, convexity of \(\Omega_\varepsilon\) and the Brascamp--Lieb
		log-concavity theorem \cite{brascamplieb} imply that
		\(x\mapsto v_\varepsilon(x,0)\) is unimodal, with its maximum at
		\(q_\varepsilon\).  If, along a subsequence,
		\(q_\varepsilon\le1/2-\eta\), then this function is nonincreasing
		on \([1/2-\eta/2,1/2]\). Let $A$ and $B$ be the left and right halves
		of this interval. Then
		\[
		\int_A v_\varepsilon(x,0)\,dx
		\ge\int_B v_\varepsilon(x,0)\,dx.
		\]
		Passing to the limit by \eqref{eq:thin-axis-convergence} contradicts
		the strict increase of $\sin(\pi x)$ on this interval, since $A$ and $B$
		have equal length. The case
		\(q_\varepsilon\ge1/2+\eta\) is analogous.  Hence
		\(q_\varepsilon\to1/2\).
	\end{proof}
	
	The final step used log-concavity only to pass from \(L^2\) control on the axis to a pointwise statement.
	
	\begin{remark}
		
		One can instead obtain uniform convergence on the axis by interior
		elliptic estimates, without using log-concavity. Extend
		$v_\varepsilon$ oddly across $x=0,1$ in the fixed central strip
		$|Y|<1/2$, where its equation is preserved weakly, and extend
		$g_\varepsilon,r_\varepsilon$ in the same way.
		For $x_0\in[0,1]$, freeze $g_\varepsilon$ at $x_0$ and rescale
		$x=x_0+\varepsilon X$. By \eqref{eq:thin-eigenvalue-sandwich}, the rescaled equation for
		$v_\varepsilon-g_\varepsilon(x_0)\chi_\varepsilon$ has a uniformly
		bounded zeroth-order coefficient and an $O(\varepsilon^2)$ right-hand side.
		The interior $H^2$ estimate and Sobolev embedding, together with
		$|g_\varepsilon(x)-g_\varepsilon(x_0)|\le C|x-x_0|^{1/2}$, give
		\[
		\sup_{x_0\in[0,1]}|r_\varepsilon(x_0,0)|
		\le C\varepsilon^{-1/2}\|r_\varepsilon\|_{L^2(R_\varepsilon^+)}
		+C\varepsilon^{1/2}+C\varepsilon^2\longrightarrow0.
		\]
		Thus \eqref{eq:thin-first-mode} implies
		$v_\varepsilon(\cdot,0)\to\sqrt2\sin(\pi x)$ uniformly on $[0,1]$.
		The strict maximum of the limit at $1/2$ then gives
		$q_\varepsilon\to1/2$ directly.
		
	\end{remark}
	
	\begingroup
	The sharpness construction is also motivated by the narrow-domain
	analysis of Li--Xie--Yang--Yao
	\cite[Theorem~1.1]{lixieyangyao2026}. For their class of convex narrow
	domains obtained by vertically compressing fixed boundary profiles,
	they prove that the
	horizontal coordinates of boundary fail points approach the set of
	maximizers of the vertical width. This suggests placing the widest
	cross section near an endpoint to influence the interior torsion
	maximum. In the construction below, that cross section is at
	\(x=\varepsilon^{1/3}\to0\), while the width perturbation is small
	enough to keep the ground-state hot spot near \(1/2\).
	Here the width profile itself depends on \(\varepsilon\) and its
	contrast tends to zero, so the earlier fail-point theorem does not
	directly imply the required interior localization. The proof below
	adapts the idea of comparison with a local strip profile, using a
	supersolution and inscribed rectangles to resolve the much smaller
	variation caused by the bump. We now prove that the diameter bound
	in Proposition~\ref{prop:upperbound} is sharp. Figure~\ref{fig:thin} shows the
	geometry of the construction.
	
	\begin{figure}[htbp]
		\centering
		\begin{tikzpicture}[xscale=9.2,yscale=38]
			\fill[black!7] (0,0.030) -- (0.12,0.057) -- (1,0.030) -- (1,-0.030)
			-- (0.12,-0.057) -- (0,-0.030) -- cycle;
			\draw[thick] (0,0.030) -- (0.12,0.057) -- (1,0.030);
			\draw[thick] (0,-0.030) -- (0.12,-0.057) -- (1,-0.030);
			\draw[thick] (0,-0.030) -- (0,0.030);
			\draw[thick] (1,-0.030) -- (1,0.030);
			\draw[dashed,black!55] (0.06,-0.082) -- (0.06,0.082);
			\draw[dashed,black!55] (0.56,-0.082) -- (0.56,0.082);
			\node[circle,fill,inner sep=0.6pt] at (0.12,0) {};
			\node[circle,fill,inner sep=0.6pt] at (0.50,0) {};
			\node[right=4pt] at (0.12,0) {\scriptsize $P_\Omega\to0$};
			\node[right=4pt] at (0.50,0) {\scriptsize $Q_\Omega\to\frac12$};
			\node at (0.12,0.100) {\scriptsize peak at $x=\varepsilon^{1/3}$};
			\node[below] at (0.31,-0.098) {\scriptsize heart window};
		\end{tikzpicture}
		\caption{A thin symmetric convex domain $\Omega_\varepsilon$: the half-width is
			$\varepsilon f_\varepsilon(x)$ with $f_\varepsilon=1+\varepsilon^{2+\sigma}\tau_\varepsilon$
			and $\tau_\varepsilon$ the tent function of the proof, which peaks at $x=\varepsilon^{1/3}$
			and vanishes at both ends; the piecewise linear profile is therefore the one used in the
			proof, while both the vertical scale and the height of the bump are exaggerated. The
			dashed lines bound the heart window \eqref{eq:heart-window}. The torsion hot spot follows
			the bump toward the left end, while the ground-state hot spot stays near the midpoint.}
		\label{fig:thin}
	\end{figure}
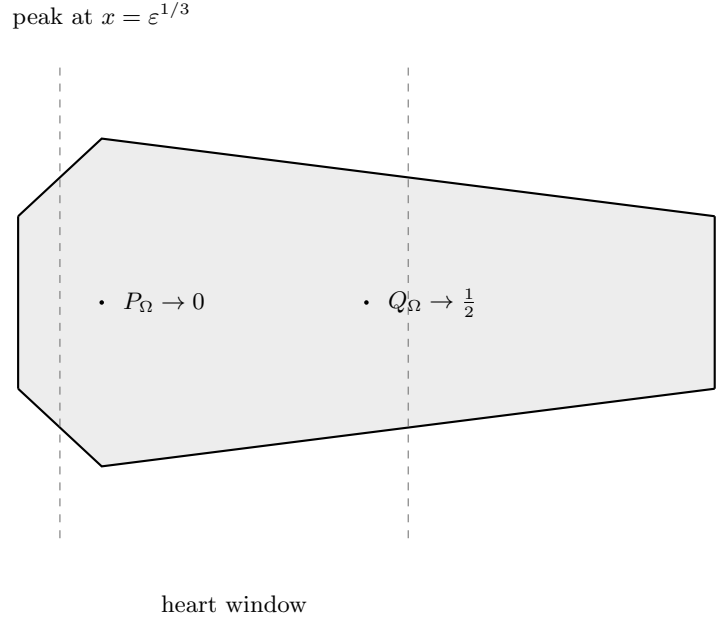
	
	\par\endgroup
	\begin{theorem}\label{thm:convex-lower-half}
		One has
		\[
		M_{\rm sym}\ge\frac12.
		\]
		More precisely, for every fixed \(\sigma>0\), there exists a family
		of symmetric convex planar domains whose relative width perturbation has size
		\(\varepsilon^{2+\sigma}\) and for which
		\[
		\frac{|P_\Omega-Q_\Omega|}
		{\operatorname{diam}(\Omega)}
		\longrightarrow\frac12.
		\]
	\end{theorem}
	
	\begin{proof}
		Fix \(\sigma>0\).
		Define
		\[
		\tau_\varepsilon(x)=
		\begin{cases}
			\dfrac{x}{\varepsilon^{1/3}}, & 0\le x\le \varepsilon^{1/3},\\[2mm]
			\dfrac{1-x}{1-\varepsilon^{1/3}}, & \varepsilon^{1/3}\le x\le 1.
		\end{cases}
		\]
		Set
		\[
		f_\varepsilon(x)=1+\varepsilon^{2+\sigma}\tau_\varepsilon(x),
		\qquad
		\Omega_\varepsilon=
		\{(x,y):0<x<1,\ |y|<\varepsilon f_\varepsilon(x)\}.
		\]
		The concavity of $f_\varepsilon$ makes $\Omega_\varepsilon$ convex.
		\begingroup
		
		For this family, \eqref{eq:heart-window} reads
		\[
		\mathcal H(\overline{\Omega_\varepsilon})
		\subset
		\left[\frac{\varepsilon^{1/3}}2,
		\frac{1+\varepsilon^{1/3}}2\right]\times\{0\}.
		\]
		We will show that the two equations select points approaching the
		opposite endpoints of this common enclosure.
		\par\endgroup
		By the vertical moving-plane argument, both maxima lie on the $x$-axis. We write
		\[
		P_{\Omega_\varepsilon}=(p_\varepsilon,0),
		\qquad
		Q_{\Omega_\varepsilon}=(q_\varepsilon,0).
		\]
		Since $\varepsilon^{2+\sigma}=o(\varepsilon^2)$,
		Lemma~\ref{lem:near-rectangle-eigen} gives 
		\[
		q_\varepsilon\to\frac12.
		\]
		
		We prove that $p_\varepsilon\to0$ by showing that, for every fixed
		$0<\eta<1$, the maximum cannot occur at $x\ge\eta$ when
		$\varepsilon$ is sufficiently small. 
		\begingroup
		The comparison is based on the infinite-strip value
		\(\varepsilon^2f_\varepsilon(x)^2/2\). Its variation generated by
		the bump is of order \(\varepsilon^{4+\sigma}\); the estimates
		below keep both the barrier correction and the effect of the
		vertical ends smaller than this scale.
		\par\endgroup
		
		Let $u_\varepsilon$ be the torsion function in $\Omega_\varepsilon$
		and set
		\[
		\Psi_\varepsilon(x,y)
		=\frac{\varepsilon^2f_\varepsilon(x)^2-y^2}
		{2(1-\varepsilon^{16/3+2\sigma})}.
		\]
		For sufficiently small $\varepsilon$, one has
		$\|f_\varepsilon'\|_\infty\le\varepsilon^{5/3+\sigma}$ and
		$1-\varepsilon^{16/3+2\sigma}>0$.
		Since $f_\varepsilon$ is continuous and piecewise affine,
		$\Psi_\varepsilon\in W^{1,\infty}(\Omega_\varepsilon)$ and is smooth
		on each side of the line $x=\varepsilon^{1/3}$. Away from this line,
		\[
		-\Delta\Psi_\varepsilon
		=\frac{1-\varepsilon^2(f_\varepsilon')^2}
		{1-\varepsilon^{16/3+2\sigma}}
		\ge1,
		\]
		because $\varepsilon^2(f_\varepsilon')^2
		\le\varepsilon^{16/3+2\sigma}$.
		At the interface,
		\[
		f_\varepsilon'(\varepsilon^{1/3}-)
		=\varepsilon^{5/3+\sigma},
		\qquad
		f_\varepsilon'(\varepsilon^{1/3}+)
		=-\frac{\varepsilon^{2+\sigma}}{1-\varepsilon^{1/3}}.
		\]
		Consequently, the jump of $\partial_x\Psi_\varepsilon$, defined as
		the right trace minus the left trace, is
		\[
		[\partial_x\Psi_\varepsilon]
		=-\frac{\varepsilon^{4+\sigma}
			f_\varepsilon(\varepsilon^{1/3})}
		{1-\varepsilon^{16/3+2\sigma}}
		\left(\varepsilon^{-1/3}
		+\frac{1}{1-\varepsilon^{1/3}}\right)<0.
		\]
		For any nonnegative $\zeta\in C_c^\infty(\Omega_\varepsilon)$,
		integration by parts separately on the two sides of the interface
		therefore gives
		\[
		\begin{aligned}
			\int_{\Omega_\varepsilon}
			\nabla\Psi_\varepsilon\cdot\nabla\zeta\,dx\,dy
			&=\int_{\Omega_\varepsilon}
			\frac{1-\varepsilon^2(f_\varepsilon')^2}
			{1-\varepsilon^{16/3+2\sigma}}\,
			\zeta\,dx\,dy\\
			&\quad-\int_{\Omega_\varepsilon\cap\{x=\varepsilon^{1/3}\}}
			[\partial_x\Psi_\varepsilon]\zeta\,dy\\
			&\ge\int_{\Omega_\varepsilon}\zeta\,dx\,dy.
		\end{aligned}
		\]
		Therefore, $-\Delta\Psi_\varepsilon\ge1$ in the distributional sense.
		
		Since
		$\Psi_\varepsilon\ge0$ on $\partial\Omega_\varepsilon$, by the maximum principle, we have $u_\varepsilon\le\Psi_\varepsilon$, and thus
		\begin{equation}\label{eq:torsion-upper-axis}
			u_\varepsilon(x,0)
			\le\frac{\varepsilon^2f_\varepsilon(x)^2}
			{2(1-\varepsilon^{16/3+2\sigma})}.
		\end{equation}
		Fix $0<\eta<1$ and let $\varepsilon\to0$ with $\eta$ fixed.
		For sufficiently small $\varepsilon$, we have
		$\varepsilon^{1/3}<\frac{\eta}{4}$. For $x\ge\eta$, we have
		$\tau_\varepsilon(x)\le\frac{1-\eta}{1-\varepsilon^{1/3}}
		=1-\eta+o(1)$.
		Since $0\le\tau_\varepsilon\le1$, one has, uniformly in $x$,
		\[
		f_\varepsilon(x)^2
		=1+2\varepsilon^{2+\sigma}\tau_\varepsilon(x)
		+O(\varepsilon^{4+2\sigma}).
		\]
		Moreover,
		\[
		\frac{1}{1-\varepsilon^{16/3+2\sigma}}
		=1+O(\varepsilon^{16/3+2\sigma})
		=1+o(\varepsilon^{2+\sigma}),
		\]
		because $\varepsilon^{16/3+2\sigma}/\varepsilon^{2+\sigma}
		=\varepsilon^{10/3+\sigma}\to0$.
		Substituting these expansions into \eqref{eq:torsion-upper-axis}
		gives
		\begin{equation}\label{eq:upper-away-left-new}
			\sup_{\eta\le x<1}u_\varepsilon(x,0)
			\le\frac{\varepsilon^2}{2}
			+\varepsilon^{4+\sigma}(1-\eta+o(1)).
		\end{equation}
		
		To obtain a lower bound at a point in $(0,\eta)$, observe that
		$f_\varepsilon$ is decreasing on
		$[\frac{\eta}{4},\frac{3\eta}{4}]$, so
		\[
		R_\varepsilon
		=(\frac{\eta}{4},\frac{3\eta}{4})\times
		\left(-\varepsilon f_\varepsilon(\frac{3\eta}{4}),
		\varepsilon f_\varepsilon(\frac{3\eta}{4})\right)
		\subset\Omega_\varepsilon.
		\]
		Domain monotonicity and Lemma~\ref{lem:long-rectangle}, applied
		at the centre of $R_\varepsilon$, give
		\[
		u_\varepsilon(\frac{\eta}{2},0)
		\ge\frac{\varepsilon^2f_\varepsilon(\frac{3\eta}{4})^2}{2}
		\left(1-C\exp\left\{
		-\frac{c\eta}{2\varepsilon f_\varepsilon(\frac{3\eta}{4})}
		\right\}\right).
		\]
		Since $1\le f_\varepsilon\le1+\varepsilon^{2+\sigma}\le2$
		for small $\varepsilon$, the exponential factor satisfies
		\[
		\exp\left\{-\frac{c\eta}
		{2\varepsilon f_\varepsilon(\frac{3\eta}{4})}\right\}
		\le e^{-c\eta/(4\varepsilon)}
		=o(\varepsilon^{2+\sigma}).
		\]
		Thus \[
		u_\varepsilon(\frac{\eta}{2},0)
		\ge\frac{\varepsilon^2f_\varepsilon(\frac{3\eta}{4})^2}{2}-o(\epsilon^{4+\sigma})
		\]
		Also,
		\[
		f_\varepsilon(\frac{3\eta}{4})
		=
		1+\varepsilon^{2+\sigma}
		\frac{1-\frac{3\eta}{4}}{1-\varepsilon^{1/3}},
		\]
		so
		\[
		\frac{\varepsilon^2f_\varepsilon(\frac{3\eta}{4})^2}{2}
		=\frac{\varepsilon^2}{2}
		+\varepsilon^{4+\sigma}
		\frac{1-\frac{3\eta}{4}}{1-\varepsilon^{1/3}}
		+O(\varepsilon^{6+2\sigma}).
		\]
		Since $\varepsilon^{1/3}\to0$, the preceding lower bound becomes
		\begin{equation}\label{eq:lower-near-left-new}
			u_\varepsilon(\frac{\eta}{2},0)
			\ge\frac{\varepsilon^2}{2}
			+\varepsilon^{4+\sigma}
			\left(1-\frac{3\eta}{4}+o(1)\right).
		\end{equation}
		Comparing \eqref{eq:upper-away-left-new} and
		\eqref{eq:lower-near-left-new} yields
		\[
		u_\varepsilon(\frac{\eta}{2},0)
		-\sup_{\eta\le x<1}u_\varepsilon(x,0)
		\ge\varepsilon^{4+\sigma}
		\left(\frac{\eta}{4}+o(1)\right)>0
		\]
		for all sufficiently small $\varepsilon$. Thus $p_\varepsilon<\eta$.
		Since $\eta$ was arbitrary, $p_\varepsilon\to0$.
		
		Together with $q_\varepsilon\to\frac{1}{2}$ and
		$\operatorname{diam}(\Omega_\varepsilon)=1+O(\varepsilon^2)$,
		this gives
		\[
		\frac{|P_{\Omega_\varepsilon}-Q_{\Omega_\varepsilon}|}
		{\operatorname{diam}(\Omega_\varepsilon)}
		=\frac{|p_\varepsilon-q_\varepsilon|}
		{\operatorname{diam}(\Omega_\varepsilon)}
		\longrightarrow\frac{1}{2}.
		\]
		
	\end{proof}
	
	The construction is robust: only the separation of scales matters, not the particular exponent.
	
	\begin{remark}
		Replace $\varepsilon^{1/3}$ by
		$\varepsilon^\gamma$ in the definition of $\tau_\varepsilon$, where
		$0<\gamma<2+\sigma/2$. For small $\varepsilon$,
		\[
		\varepsilon^2\|f_\varepsilon'\|_\infty^2
		=\varepsilon^{6+2\sigma-2\gamma}
		=o(\varepsilon^{2+\sigma}),
		\]
		because $4+\sigma-2\gamma>0$. The same barrier argument applies with
		$1-\varepsilon^{6+2\sigma-2\gamma}$ in place of
		$1-\varepsilon^{16/3+2\sigma}$ in its denominator. The correction
		remains smaller than the relative width perturbation, so the preceding
		comparison still gives $p_\varepsilon\to0$. The eigenfunction lemma
		also applies, since $\|f_\varepsilon-1\|_\infty
		=\varepsilon^{2+\sigma}=o(\varepsilon^2)$.
		
		The size $\varepsilon^{2+\sigma}$ of that perturbation has a different
		role for the ground state. The first transverse eigenvalue satisfies
		\[
		\frac{\pi^2}{4\varepsilon^2f_\varepsilon(x)^2}
		=\frac{\pi^2}{4\varepsilon^2}
		-\frac{\pi^2}{2}\varepsilon^\sigma\tau_\varepsilon(x)
		+O(\varepsilon^{2+2\sigma}),
		\]
		uniformly in $x$. Its variation is therefore negligible compared
		with the order-one longitudinal energy, which leads to the limiting
		axial profile $\sin(\pi x)$.
	\end{remark}
	For the tent profile of the preceding theorem, the convergence rates of
	$p_\varepsilon$ and $q_\varepsilon$ can be made explicit.

	\begin{proposition}
		\label{prop:tent-hotspot-rates}
		Fix $\sigma>0$, and consider the domains
		\[
		\Omega_\varepsilon
		=\{(x,y):0<x<1,\ |y|<\varepsilon f_\varepsilon(x)\},
		\qquad
		f_\varepsilon=1+\varepsilon^{2+\sigma}\tau_\varepsilon,
		\]
		where
		\[
		\tau_\varepsilon(x)=
		\begin{cases}
			\dfrac{x}{\varepsilon^{1/3}},&0\le x\le\varepsilon^{1/3},\\[2mm]
			\dfrac{1-x}{1-\varepsilon^{1/3}},&\varepsilon^{1/3}\le x\le1.
		\end{cases}
		\]
		Let $u_\varepsilon$ be the torsion function, and denote the torsion
		and ground-state hot spots by
		\[
		P_{\Omega_\varepsilon}=(p_\varepsilon,0),
		\qquad Q_{\Omega_\varepsilon}=(q_\varepsilon,0).
		\]
		There exist $C_\sigma>0$ and $\varepsilon_\sigma\in(0,1/2)$ such that,
		for $0<\varepsilon<\varepsilon_\sigma$,
		\begin{equation}\label{eq:tent-quantitative-hotspots}
			|p_\varepsilon-\varepsilon^{1/3}|
			\le C_\sigma\varepsilon\log\frac1\varepsilon,
			\qquad
			|q_\varepsilon-1/2|\le C_\sigma\varepsilon^{\sigma/5}.
		\end{equation}
		In particular, $p_\varepsilon/\varepsilon^{1/3}\to1$.
	\end{proposition}
	
	\begin{proof}
		All estimates below are for sufficiently small $\varepsilon$.
		Constants denoted by $C$ may change from line to line and may depend
		on the fixed parameter $\sigma$.
		
		\emph{Step 1: localization of the torsion hot spot.}
		The supersolution estimate established in the proof of
		Theorem~\ref{thm:convex-lower-half} gives
		\begin{equation}\label{eq:rates-torsion-upper}
			u_\varepsilon(x,0)
			\le\frac{\varepsilon^2 f_\varepsilon(x)^2}
			{2(1-\varepsilon^{16/3+2\sigma})}
			\le\frac{\varepsilon^2 f_\varepsilon(x)^2}{2}
			+C\varepsilon^{22/3+2\sigma},
		\end{equation}
		uniformly for $0<x<1$. The second inequality uses
		$1\le f_\varepsilon\le2$.
		
		Let
		\[
		\ell=K_\sigma\varepsilon\log\frac1\varepsilon,
		\]
		where $K_\sigma>0$ will be chosen below. The rectangle
		\[
		\begin{aligned}
			R_\varepsilon
			&=(\varepsilon^{1/3},\varepsilon^{1/3}+2\ell)\\
			&\quad\times
			\bigl(-\varepsilon f_\varepsilon(\varepsilon^{1/3}+2\ell),
			\varepsilon f_\varepsilon(\varepsilon^{1/3}+2\ell)\bigr)
		\end{aligned}
		\]
		is contained in $\Omega_\varepsilon$, since $f_\varepsilon$ is
		decreasing to the right of $\varepsilon^{1/3}$.
		Its centre is $(\varepsilon^{1/3}+\ell,0)$, and its half-width lies
		between $\varepsilon$ and $2\varepsilon$.
		By domain monotonicity and Lemma~\ref{lem:long-rectangle},
		\begin{equation}\label{eq:rates-rectangle-lower}
			\begin{aligned}
				u_\varepsilon(\varepsilon^{1/3}+\ell,0)
				&\ge\frac{\varepsilon^2
					f_\varepsilon(\varepsilon^{1/3}+2\ell)^2}{2}
				-C\varepsilon^2e^{-c\ell/\varepsilon},
			\end{aligned}
		\end{equation}
		with an absolute constant $c>0$.
		Choose $K_\sigma$ so that $cK_\sigma>3+\sigma$. Then
		\[
		\varepsilon^2e^{-c\ell/\varepsilon}
		=\varepsilon^{2+cK_\sigma}
		=o(\varepsilon^{5+\sigma}).
		\]
		On the other hand, the definition of the tent gives
		\[
		f_\varepsilon(\varepsilon^{1/3}+2\ell)
		=1+\varepsilon^{2+\sigma}
		-\frac{2\varepsilon^{2+\sigma}\ell}{1-\varepsilon^{1/3}}.
		\]
		Consequently,
		\[
		\begin{aligned}
			0&\le
			\frac{\varepsilon^2(1+\varepsilon^{2+\sigma})^2}{2}
			-\frac{\varepsilon^2
				f_\varepsilon(\varepsilon^{1/3}+2\ell)^2}{2}\\
			&\le C\varepsilon^{4+\sigma}\ell
			\le C_\sigma\varepsilon^{5+\sigma}\log\frac1\varepsilon.
		\end{aligned}
		\]
		Combining this with \eqref{eq:rates-rectangle-lower} yields
		\begin{equation}\label{eq:rates-torsion-max-lower}
			u_\varepsilon(p_\varepsilon,0)
			\ge\frac{\varepsilon^2(1+\varepsilon^{2+\sigma})^2}{2}
			-C_\sigma\varepsilon^{5+\sigma}\log\frac1\varepsilon.
		\end{equation}
		Since
		\[
		\frac{\varepsilon^{22/3+2\sigma}}
		{\varepsilon^{5+\sigma}\log(1/\varepsilon)}
		=\frac{\varepsilon^{7/3+\sigma}}{\log(1/\varepsilon)}\longrightarrow0,
		\]
		comparison with \eqref{eq:rates-torsion-upper} at $x=p_\varepsilon$
		gives
		\[
		\frac{\varepsilon^2}{2}
		\left((1+\varepsilon^{2+\sigma})^2
		-f_\varepsilon(p_\varepsilon)^2\right)
		\le C_\sigma\varepsilon^{5+\sigma}\log\frac1\varepsilon.
		\]
		The expression on the left equals
		\[
		\varepsilon^{4+\sigma}
		\bigl(1-\tau_\varepsilon(p_\varepsilon)\bigr)
		\left(1+\frac{\varepsilon^{2+\sigma}}2
		\bigl(1+\tau_\varepsilon(p_\varepsilon)\bigr)\right).
		\]
		The last factor is at least $1$, and hence
		\[
		1-\tau_\varepsilon(p_\varepsilon)
		\le C_\sigma\varepsilon\log\frac1\varepsilon.
		\]
		Finally,
		\[
		1-\tau_\varepsilon(x)=
		\begin{cases}
			\dfrac{\varepsilon^{1/3}-x}{\varepsilon^{1/3}},
			&0\le x\le\varepsilon^{1/3},\\[2mm]
			\dfrac{x-\varepsilon^{1/3}}{1-\varepsilon^{1/3}},
			&\varepsilon^{1/3}\le x\le1,
		\end{cases}
		\]
		so $|x-\varepsilon^{1/3}|\le1-\tau_\varepsilon(x)$.
		This proves the first estimate in \eqref{eq:tent-quantitative-hotspots}.
		It also gives
		\[
		\left|\frac{p_\varepsilon}{\varepsilon^{1/3}}-1\right|
		\le C_\sigma\varepsilon^{2/3}\log\frac1\varepsilon\longrightarrow0.
		\]
		
		\emph{Step 2: a quantitative approximation of the ground-state profile.} The scheme is that of
		Lemma~\ref{lem:near-rectangle-eigen}; we repeat it because the rates, not only the limits, are
		needed here.
		Let $\phi_\varepsilon$ be the positive first Dirichlet eigenfunction
		in $\Omega_\varepsilon$, with eigenvalue $\lambda_\varepsilon$.
		Normalize the rescaled function
		\[
		v_\varepsilon(x,Y)=\phi_\varepsilon(x,\varepsilon Y),
		\qquad
		D_\varepsilon=\{(x,Y):0<x<1,\ |Y|<f_\varepsilon(x)\},
		\]
		so that $\int_{D_\varepsilon}v_\varepsilon^2=1$, and extend it by zero to
		\[
		R_\varepsilon^+
		=(0,1)\times(-1-\varepsilon^{2+\sigma},1+\varepsilon^{2+\sigma}).
		\]
		Since $(0,1)\times(-1,1)\subset D_\varepsilon\subset R_\varepsilon^+$,
		domain monotonicity gives
		\[
		\pi^2+\frac{\pi^2}{4\varepsilon^2(1+\varepsilon^{2+\sigma})^2}
		\le\lambda_\varepsilon
		\le\pi^2+\frac{\pi^2}{4\varepsilon^2}.
		\]
		Write
		\[
		\mu_\varepsilon
		=\frac{\pi^2}{4(1+\varepsilon^{2+\sigma})^2}.
		\]
		The difference between the two rectangle eigenvalues satisfies
		\[
		\frac{\pi^2}{4\varepsilon^2}
		\left(1-\frac{1}{(1+\varepsilon^{2+\sigma})^2}\right)
		=\frac{\pi^2}{4}\varepsilon^\sigma
		\frac{2+\varepsilon^{2+\sigma}}{(1+\varepsilon^{2+\sigma})^2}
		\le C\varepsilon^\sigma.
		\]
		Hence
		\begin{equation}\label{eq:rates-eigenvalue-excess}
			0\le\lambda_\varepsilon-\pi^2
			-\varepsilon^{-2}\mu_\varepsilon
			\le C\varepsilon^\sigma.
		\end{equation}
		The zero extension belongs to $H_0^1(R_\varepsilon^+)$ and satisfies
		\[
		\lambda_\varepsilon
		=\int_{R_\varepsilon^+}|\partial_xv_\varepsilon|^2
		+\varepsilon^{-2}\int_{R_\varepsilon^+}
		|\partial_Yv_\varepsilon|^2.
		\]
		By the one-dimensional Poincar\'e inequalities and normalization,
		the two energy excesses are nonnegative. Thus
		\eqref{eq:rates-eigenvalue-excess} implies
		\begin{equation}\label{eq:rates-directional-excesses}
			\begin{aligned}
				0&\le\int_{R_\varepsilon^+}|\partial_xv_\varepsilon|^2-\pi^2
				\le C\varepsilon^\sigma,\\
				0&\le\int_{R_\varepsilon^+}|\partial_Yv_\varepsilon|^2
				-\mu_\varepsilon
				\le C\varepsilon^{2+\sigma}.
			\end{aligned}
		\end{equation}
		
		Let
		\[
		\chi_\varepsilon(Y)
		=(1+\varepsilon^{2+\sigma})^{-1/2}
		\cos\frac{\pi Y}{2(1+\varepsilon^{2+\sigma})}
		\]
		be the normalized first vertical eigenfunction, and set
		\[
		g_\varepsilon(x)
		=\int_{-1-\varepsilon^{2+\sigma}}^{1+\varepsilon^{2+\sigma}}
		v_\varepsilon(x,Y)\chi_\varepsilon(Y)\,dY,
		\qquad
		r_\varepsilon=v_\varepsilon-g_\varepsilon\chi_\varepsilon.
		\]
		Then $g_\varepsilon\in H_0^1(0,1)$, $g_\varepsilon\ge0$,
		$r_\varepsilon\in H_0^1(R_\varepsilon^+)$, and
		$r_\varepsilon(x,\cdot)$ is orthogonal to $\chi_\varepsilon$.
		In the following formulas, norms of functions of $x$ alone are taken
		on $(0,1)$ and norms of functions of $(x,Y)$ on $R_\varepsilon^+$.
		Orthogonality, and integration by parts using
		$-\chi_\varepsilon''=\mu_\varepsilon\chi_\varepsilon$, give
		\[
		\|g_\varepsilon\|_2^2+\|r_\varepsilon\|_2^2=1,
		\qquad
		\int_{R_\varepsilon^+}|\partial_Yv_\varepsilon|^2
		=\mu_\varepsilon\|g_\varepsilon\|_2^2
		+\|\partial_Yr_\varepsilon\|_2^2.
		\]
		The second vertical eigenvalue is $4\mu_\varepsilon$. Therefore
		\[
		3\mu_\varepsilon\|r_\varepsilon\|_2^2
		\le\int_{R_\varepsilon^+}|\partial_Yv_\varepsilon|^2
		-\mu_\varepsilon
		\le C\varepsilon^{2+\sigma}.
		\]
		Since $\mu_\varepsilon$ stays bounded above and away from zero, the
		preceding identities also give
		\begin{equation}\label{eq:rates-vertical-remainder}
			\begin{aligned}
				\|r_\varepsilon\|_2^2&\le C\varepsilon^{2+\sigma},\\
				\|\partial_Yr_\varepsilon\|_2^2
				&=\left(\int_{R_\varepsilon^+}|\partial_Yv_\varepsilon|^2
				-\mu_\varepsilon\right)
				+\mu_\varepsilon\|r_\varepsilon\|_2^2
				\le C\varepsilon^{2+\sigma}.
			\end{aligned}
		\end{equation}
		
		Moreover, differentiating the projection in $x$ and applying
		Cauchy--Schwarz yields
		\[
		\|g_\varepsilon'\|_2^2
		\le\int_{R_\varepsilon^+}|\partial_xv_\varepsilon|^2
		\le\pi^2+C\varepsilon^\sigma.
		\]
		Using $\|g_\varepsilon\|_2^2=1-\|r_\varepsilon\|_2^2$, we obtain
		\begin{equation}\label{eq:rates-horizontal-defect}
			0\le\|g_\varepsilon'\|_2^2
			-\pi^2\|g_\varepsilon\|_2^2\le C\varepsilon^\sigma.
		\end{equation}
		To quantify the corresponding near-equality case, write
		\[
		e(x)=\sqrt2\sin(\pi x),
		\qquad
		g_\varepsilon=c_\varepsilon e+w_\varepsilon,
		\qquad
		\int_0^1w_\varepsilon e\,dx=0.
		\]
		Here $c_\varepsilon=\int_0^1g_\varepsilon e\,dx\ge0$.
		Since the second Dirichlet eigenvalue of $(0,1)$ is $4\pi^2$,
		\[
		\begin{aligned}
			\|g_\varepsilon'\|_2^2-\pi^2\|g_\varepsilon\|_2^2
			&=\|w_\varepsilon'\|_2^2-\pi^2\|w_\varepsilon\|_2^2\\
			&\ge\frac34\|w_\varepsilon'\|_2^2.
		\end{aligned}
		\]
		Thus $\|w_\varepsilon\|_{H_0^1(0,1)}\le C\varepsilon^{\sigma/2}$.
		Normalization gives
		\[
		0\le1-c_\varepsilon^2
		=\|w_\varepsilon\|_2^2+\|r_\varepsilon\|_2^2
		\le C\varepsilon^\sigma,
		\]
		and hence $|1-c_\varepsilon|\le C\varepsilon^\sigma$.
		It follows that
		\begin{equation}\label{eq:rates-horizontal-mode}
			\|g_\varepsilon-e\|_{H_0^1(0,1)}
			\le C\varepsilon^{\sigma/2}.
		\end{equation}
		
		To pass to the axis, use the zero boundary value of $r_\varepsilon$
		at $Y=-1-\varepsilon^{2+\sigma}$. For almost every $x$,
		\[
		r_\varepsilon(x,0)
		=\int_{-1-\varepsilon^{2+\sigma}}^0
		\partial_Yr_\varepsilon(x,Y)\,dY.
		\]
		Cauchy--Schwarz and \eqref{eq:rates-vertical-remainder} imply
		\[
		\|r_\varepsilon(\cdot,0)\|_{L^2(0,1)}^2
		\le(1+\varepsilon^{2+\sigma})
		\|\partial_Yr_\varepsilon\|_2^2
		\le C\varepsilon^{2+\sigma}.
		\]
		Since $\chi_\varepsilon(0)=1+O(\varepsilon^{2+\sigma})$,
		\eqref{eq:rates-horizontal-mode} therefore yields
		\begin{equation}\label{eq:rates-axis-profile}
			\begin{aligned}
				\|v_\varepsilon(\cdot,0)-e\|_{L^2(0,1)}
				&\le |\chi_\varepsilon(0)|\,
				\|g_\varepsilon-e\|_{L^2(0,1)}
				+|\chi_\varepsilon(0)-1|\\
				&\quad+\|r_\varepsilon(\cdot,0)\|_{L^2(0,1)}
				\le C\varepsilon^{\sigma/2}.
			\end{aligned}
		\end{equation}
		
		\emph{Step 3: localization of the ground-state hot spot.}
		Convexity of $\Omega_\varepsilon$ and the Brascamp--Lieb log-concavity
		theorem imply that $x\mapsto v_\varepsilon(x,0)$ is nondecreasing
		before $q_\varepsilon$ and nonincreasing after $q_\varepsilon$.
		Set $d=|q_\varepsilon-1/2|$. There is nothing to prove if $d=0$.
		Suppose first that $q_\varepsilon=1/2-d$, and put
		\[
		A=(1/2-d/2,1/2-d/4),
		\qquad B=(1/2-d/4,1/2).
		\]
		Both intervals lie to the right of $q_\varepsilon$, so
		\[
		\int_Bv_\varepsilon(x,0)\,dx
		\le\int_Av_\varepsilon(x,0)\,dx.
		\]
		For the limiting profile $e$, direct integration gives
		\[
		\begin{aligned}
			\int_Be\,dx-\int_Ae\,dx
			&=\frac{\sqrt2}{\pi}
			\left(2\sin\frac{\pi d}{4}-\sin\frac{\pi d}{2}\right)\\
			&=\frac{2\sqrt2}{\pi}\sin\frac{\pi d}{4}
			\left(1-\cos\frac{\pi d}{4}\right)
			\ge cd^3,
		\end{aligned}
		\]
		with an absolute $c>0$ for $0<d\le1/2$.
		The last inequality follows from
		$\sin t\ge2t/\pi$ for $0\le t\le\pi/2$ and
		$1-\cos t=2\sin^2(t/2)$.
		Consequently, Cauchy--Schwarz gives
		\[
		\begin{aligned}
			cd^3
			&\le\int_Be\,dx-\int_Ae\,dx\\
			&\le\int_B\bigl(e-v_\varepsilon(\cdot,0)\bigr)\,dx
			-\int_A\bigl(e-v_\varepsilon(\cdot,0)\bigr)\,dx\\
			&\le\sqrt{|A|+|B|}\,
			\|e-v_\varepsilon(\cdot,0)\|_{L^2(0,1)}\\
			&=\sqrt{d/2}\,
			\|e-v_\varepsilon(\cdot,0)\|_{L^2(0,1)}.
		\end{aligned}
		\]
		Using \eqref{eq:rates-axis-profile} and dividing by $\sqrt d$, we obtain
		\[
		d^{5/2}\le C_\sigma\varepsilon^{\sigma/2},
		\qquad d\le C_\sigma\varepsilon^{\sigma/5}.
		\]
		If $q_\varepsilon>1/2$, the same argument uses the reflected intervals
		to the right of $1/2$; the identity $e(1-x)=e(x)$ gives the same
		integral difference. This proves the second estimate in
		\eqref{eq:tent-quantitative-hotspots}.
	\end{proof}

	\begin{proof}[Proof of Theorem~\ref{thm:main}]
		Proposition~\ref{prop:upperbound} and
		Theorem~\ref{thm:convex-lower-half} give $M_{\rm sym}=1/2$, and the strict
		inequality is the second part of Proposition~\ref{prop:upperbound}.
		The refined estimates follow from Proposition~\ref{prop:upperbound}.
	\end{proof}

	\section{Sharp isosceles triangles}
	\label{sec:triangles}
	
	This section discusses a related setting in which the relevant scales are even clearer. Let
	\[
	\Omega_h=\{(x,y):0<x<1,\ |y|<hx\},
	\qquad h>0.
	\]
	This is an isosceles triangle with vertex at the origin, axis of symmetry the $x$-axis, and base $x=1$. Its aperture is
	\[
	\alpha=2\arctan h.
	\]
	Let $u_h$ be the torsion function in $\Omega_h$:
	\[
	-\Delta u_h=1\quad\text{in }\Omega_h,
	\qquad
	u_h=0\quad\text{on }\partial\Omega_h,
	\]
	and let $\phi_h$ be the first Dirichlet eigenfunction:
	\[
	-\Delta \phi_h=\lambda_h\phi_h\quad\text{in }\Omega_h,
	\qquad
	\phi_h=0\quad\text{on }\partial\Omega_h,
	\qquad
	\phi_h>0,
	\qquad
	\|\phi_h\|_{L^2(\Omega_h)}=1.
	\]
	Since $\Omega_h$ is symmetric with respect to the $x$-axis and every vertical section
	\[
	\Omega_h\cap\{x=\mathrm{const.}\}
	\]
	is an interval symmetric about $y=0$, the moving-plane method implies that $u_h$ and $\phi_h$ are even in $y$ and strictly decreasing in $y>0$. Hence both maximum points lie on the axis of symmetry. Denote their horizontal coordinates by
	\[
	x_T(h),\qquad x_E(h).
	\]
	We prove that, for sufficiently sharp triangles, the torsion hot spot lies closer to the base than the first-eigenfunction hot spot.
	
	\begin{theorem}
		\label{thm:thin-triangle}
		There exists $h_0>0$ such that for all $0<h<h_0$,
		\[
		x_T(h)>x_E(h).
		\]
	\end{theorem}
	
	The proof uses Lemma~\ref{lem:long-rectangle}, established in
	Section~\ref{sec:convex-diameter}, and the following Airy asymptotic
	for the first eigenfunction in a sharp thin triangle.
	
	\begin{lemma}
		\label{airy}
		Let $x_E(h)$ be the horizontal coordinate of the maximum point of the
		first Dirichlet eigenfunction in
		\[
		\Omega_h=\{(x,y):0<x<1,\ |y|<hx\}.
		\]
		Let \(a_1<0\) and \(a_1'<0\) be the first zeros of
		\(\operatorname{Ai}\) and \(\operatorname{Ai}'\), respectively.  Then
		\[
		1-x_E(h)
		=
		(a_1'-a_1)
		\left(\frac{\pi^2}{2}\right)^{-1/3}h^{2/3}
		+o(h^{2/3}).
		\]
		In particular, there exist \(c_0>0\) and \(h_1>0\)
		such that \(1-x_E(h)\ge c_0h^{2/3}\) for \(0<h<h_1\).
	\end{lemma}
	
	\begin{proof}
		Set
		\[
		u=\pi\sqrt2(x-1),
		\qquad
		v=\frac{\pi\sqrt2}{h}y.
		\]
		This affine change of variables maps $\Omega_h$ onto the fixed triangle
		\[
		\operatorname{Tri}
		=
		\{(u,v):-\pi\sqrt2<u<0,\ |v|<u+\pi\sqrt2\},
		\]
		and
		\[
		-\Delta_{x,y}
		=
		\frac{2\pi^2}{h^2}
		\left(-h^2\partial_u^2-\partial_v^2\right).
		\]
		
		With the normalization above, define the pullback
		\[
		\psi_h(u,v)
		=
		\left(\frac{h}{2\pi^2}\right)^{1/2}
		\phi_h\left(
		1+\frac{u}{\pi\sqrt2},
		\frac{hv}{\pi\sqrt2}
		\right).
		\]
		The Jacobian identity
		\(du\,dv=(\frac{2\pi^2}{h})\,dx\,dy\) gives
		\(\|\psi_h\|_{L^2(\operatorname{Tri})}=1\).
		
		Let \(L_h=-h^2\partial_u^2-\partial_v^2\) on
		\(\operatorname{Tri}\), and denote its first two eigenvalues by
		\(\lambda_{1,h}<\lambda_{2,h}\).  For the first Airy branch,
		Dauge--Raymond \cite[Theorem~2.4, Proposition~5.1 and
		formula~(5.14)]{DaugeRaymond2012} construct, to arbitrary order, a
		normalized quasimode \(q_h\) and a quasi-eigenvalue \(\beta_h\) such that
		\[
		\|(L_h-\beta_h)q_h\|_{L^2}=o(h^{2/3}),
		\qquad
		|\beta_h-\lambda_{1,h}|=o(h^{2/3}),
		\qquad
		\lambda_{2,h}-\lambda_{1,h}\asymp h^{2/3}.
		\]
		If \(P_{1,h}\) is the orthogonal projection onto the first eigenspace,
		then the spectral theorem gives
		\[
		\|(I-P_{1,h})q_h\|_{L^2}
		\le
		\frac{\|(L_h-\beta_h)q_h\|_{L^2}}
		{\operatorname{dist}(\beta_h,
			\sigma(L_h)\setminus\{\lambda_{1,h}\})}
		=o(1).
		\]
		The first eigenspace is spanned by the positive normalized function
		\(\psi_h\).  Choose the sign of \(q_h\) so that
		\(\langle q_h,\psi_h\rangle\ge0\).  Since \(\|q_h\|_2=1\), the
		preceding projection estimate implies
		\[
		\|q_h-\psi_h\|_{L^2(\operatorname{Tri})}=o(1).
		\]
		
		Set
		\[
		D_h
		=
		\left\{
		(z,t):
		0<z<h^{-2/3},\
		|t|<1-h^{2/3}z
		\right\},
		\qquad
		D=(0,\infty)\times(-1,1).
		\]
		The leading term in formula~(5.14) gives positive constants
		\(A_h\asymp h^{1/3}\) such that
		\[
		A_hq_h(-\pi\sqrt2\,h^{2/3}z,\pi\sqrt2\,t)
		\longrightarrow
		\Phi(z,t)
		:=
		\operatorname{Ai}\left(
		\left(\frac{\pi^2}{2}\right)^{1/3}z+a_1
		\right)
		\cos\frac{\pi t}{2}
		\]
		in \(L^2_{\mathrm{loc}}(D)\), where \(a_1<0\) is the first zero of
		\(\operatorname{Ai}\).  Indeed, the transverse variable in that formula
		is
		\[
		\frac{v}{u+\pi\sqrt2}
		=
		\frac{t}{1-h^{2/3}z}
		\longrightarrow t,
		\]
		and the Airy scaling constant is
		\(\frac{\pi\sqrt2}{(4\pi\sqrt2)^{1/3}}=(\frac{\pi^2}{2})^{1/3}\).
		
		Define
		\[
		\Phi_h(z,t)
		:=
		A_h\psi_h(-\pi\sqrt2\,h^{2/3}z,\pi\sqrt2\,t)
		=
		c_h\phi_h(1-h^{2/3}z,ht),
		\]
		where
		\(c_h=A_h(\frac{h}{2\pi^2})^{1/2}\asymp h^{5/6}\).
		For every compact \(K\Subset D\), the change of variables
		\(du\,dv=2\pi^2h^{2/3}\,dz\,dt\) yields
		\[
		\begin{split}
			&\|A_h(\psi_h-q_h)
			(-\pi\sqrt2\,h^{2/3}z,\pi\sqrt2\,t)\|_{L^2(K)}^2\\
			&\qquad\le
			\frac{A_h^2}{2\pi^2h^{2/3}}
			\|\psi_h-q_h\|_{L^2(\operatorname{Tri})}^2
			=o(1).
		\end{split}
		\]
		Consequently,
		\[
		\Phi_h\longrightarrow\Phi
		\qquad\text{in }L^2_{\mathrm{loc}}(D).
		\]
		Every compact subset of $D$ is contained in $D_h$ for all sufficiently
		small $h$.  Since $\Omega_h$ is convex, the Brascamp--Lieb theorem
		\cite{brascamplieb} shows that $\phi_h$ is log-concave.  The map
		\[
		(z,t)\longmapsto(1-h^{2/3}z,ht)
		\]
		is affine, so $\Phi_h$ is log-concave on $D_h$.
		
		We claim that the preceding $L^2_{\mathrm{loc}}$ convergence is locally
		uniform.  Let $O\Subset D$ be open and convex and set
		\[
		G_h=-\log\Phi_h.
		\]
		For small $h$, the functions $G_h$ are finite and convex on $O$.
		From any sequence $h\to0$, the $L^2(O)$ convergence yields a subsequence
		for which $\Phi_h\to\Phi$ almost everywhere in $O$.  Since $\Phi>0$ in
		$D$, it follows that
		\[
		G_h\longrightarrow-\log\Phi
		\]
		on a set of full measure, hence on a dense subset of $O$.
		The compact-convergence theorem for finite convex functions
		\cite[Theorem~10.8]{rockafellar} then gives local uniform convergence of
		$G_h$ and hence of $\Phi_h$.  As every sequence has a subsequence with
		the same limit, we conclude that
		\[
		\Phi_h\longrightarrow\Phi
		\qquad\text{locally uniformly in }D.
		\]
		
		Let
		\[
		F(z)=
		\operatorname{Ai}\left(
		\left(\frac{\pi^2}{2}\right)^{1/3}z+a_1
		\right).
		\]
		If
		\[
		z_0=(a_1'-a_1)
		\left(\frac{\pi^2}{2}\right)^{-1/3},
		\]
		then the standard zero interlacing for \(\operatorname{Ai}\) and
		\(\operatorname{Ai}'\) shows that \(F\) is strictly increasing on
		\((0,z_0)\) and strictly decreasing on \((z_0,\infty)\).
		
		Set
		\[
		z_h=\frac{1-x_E(h)}{h^{2/3}}.
		\]
		The function $z\mapsto\Phi_h(z,0)$ is log-concave and attains
		its maximum at $z_h$. For fixed $0<\eta<z_0$, the strict
		maximum of $F$ and local uniform convergence give
		\[
		\Phi_h(z_0-\eta,0)<\Phi_h(z_0,0),\qquad
		\Phi_h(z_0+\eta,0)<\Phi_h(z_0,0)
		\]
		for all sufficiently small $h$.
		If $z_h\le z_0-\eta$, the profile is nonincreasing between
		$z_0-\eta$ and $z_0$, contradicting the first inequality.
		If $z_h\ge z_0+\eta$, it is nondecreasing between $z_0$ and
		$z_0+\eta$, contradicting the second. Therefore
		$|z_h-z_0|<\eta$, and $z_h\to z_0$ follows.
	\end{proof}

	We now prove Theorem~\ref{thm:thin-triangle}.
	\begin{proof}
		For $0<h<1$, the function
		\[
		\Psi(x,y)=\frac{h^2x^2-y^2}{2(1-h^2)}
		\]
		satisfies $-\Delta\Psi=1$, vanishes on the slanted sides, and
		is nonnegative on the base. The comparison principle gives
		\[
		u_h(x,0)\le\frac{h^2x^2}{2(1-h^2)}.
		\]
		Set $R=Ah\log(1/h)$, where $A>1/c$ and $c$ is the constant
		in Lemma~\ref{lem:long-rectangle}. Since $h^2=o(R)$,
		\[
		\sup_{0<x\le1-2R}u_h(x,0)
		\le\frac{h^2(1-2R)^2}{2(1-h^2)}
		=\frac{h^2}{2}\bigl(1-4R+o(R)\bigr).
		\]
		On the other hand,
		\[
		Q_h=(1-R,1)\times(-h(1-R),h(1-R))\subset\Omega_h.
		\]
		Domain monotonicity and Lemma~\ref{lem:long-rectangle} yield
		\[
		u_h(1-\frac{R}{2},0)
		\ge\frac{h^2(1-R)^2}{2}
		\left(1-Ce^{-cR/[h(1-R)]}\right).
		\]
		Our choice of $A$ ensures
		$e^{-cR/[h(1-R)]}\le h^{cA}=o(R)$, so
		\[
		u_h(1-\frac{R}{2},0)
		\ge\frac{h^2}{2}\bigl(1-2R+o(R)\bigr).
		\]
		This is strictly larger than the preceding upper bound.
		By axial localization, $x_T(h)>1-2R$, and hence
		\[
		1-x_T(h)=O\bigl(h\log(\frac{1}{h})\bigr)=o(h^{2/3}).
		\]
		Lemma~\ref{airy} gives $1-x_E(h)\ge c_0h^{2/3}$ for small $h$.
		Consequently $x_T(h)>x_E(h)$, as claimed.
	\end{proof}

	For comparison, consider the right-isosceles triangle. In \cite{bensonlaugesen}, the right isosceles triangle
	is written as
	\[
	\widetilde\Omega
	=
	\{(\widetilde x,y):0<\widetilde x<1,\ |y|<1-\widetilde x\}.
	\]
	For \(h=1\), our triangle \(\Omega_1\) is obtained from
	\(\widetilde\Omega\) by the reflection
	\[
	(x,y)=(1-\widetilde x,y).
	\]
	The values reported in \cite{bensonlaugesen},
	\[
	\widetilde x_T\approx0.39168,
	\qquad
	\widetilde x_E\approx0.39183,
	\]
	therefore become
	\[
	x_T\approx0.60832>0.60817\approx x_E.
	\]
	We do not address the ordering for general $h>0$.

	\appendix
	
	\section{Proof of the half-disk formula}
	\label{app:half-disk}
	
	\begin{proof}
		For \(\alpha=\pi\), we have \(\mu_k=2k+1\). Set
		\(z=x+iy=re^{i\theta}\) and
		\[
		G(z)=\sum_{k=0}^{\infty}
		\frac{(-1)^kz^{2k+1}}
		{(2k+1)((2k+1)^2-4)}.
		\]
		The separated-variable representation gives
		\[
		u(r,\theta)
		=\frac4\pi\left\{
		r^2\operatorname{Re}G(e^{i\theta})
		-\operatorname{Re}G(z)
		\right\}.
		\]
		
		For \(|z|<1\),
		\[
		\arctan z
		=\sum_{k=0}^{\infty}\frac{(-1)^kz^{2k+1}}{2k+1}.
		\]
		
		Using the partial fraction identity
		\[
		\frac1{n(n^2-4)}
		=-\frac1{4n}+\frac1{8(n-2)}+\frac1{8(n+2)},
		\]
		we obtain
		\[
		\begin{aligned}
			G(z)
			&=-\frac14\arctan z
			+\frac18(-z-z^2\arctan z)
			+\frac18z^{-2}(z-\arctan z)\\
			&=-\frac18\left((z+z^{-1})^2\arctan z
			+z-z^{-1}\right),
			\qquad 0<|z|<1.
		\end{aligned}
		\]
		With the principal logarithm,
		\[
		\arctan z
		=\frac1{2i}\operatorname{Log}\frac{1+iz}{1-iz},
		\qquad
		\frac{1+iz}{1-iz}
		=\frac{1-x^2-y^2+2ix}{x^2+(1+y)^2}.
		\]
		For \(z=e^{i\theta}\), \(|\theta|<\pi/2\), the latter
		quotient lies on the positive imaginary axis. Thus
		\(\operatorname{Re}\arctan(e^{i\theta})=\pi/4\).
		By continuity,
		\[
		\operatorname{Re}G(e^{i\theta})
		=-\frac{\pi}{8}\cos^2\theta,
		\qquad
		r^2\operatorname{Re}G(e^{i\theta})
		=-\frac{\pi}{8}x^2.
		\]
		
		For \(x>0\) and \(r<1\), the logarithmic representation gives
		\[
		\arctan z
		=\frac12\arctan\frac{2x}{1-x^2-y^2}
		+\frac{i}{4}
		\log\frac{x^2+(1+y)^2}{x^2+(1-y)^2}.
		\]
		Taking real parts in the expression for \(G(z)\) and
		substituting into the representation of \(u\), we find
		\[
		\begin{aligned}
			u(x,y)
			={}&-\frac{x^2}{2}
			+\frac{2+(x^2-y^2)(1+r^{-4})}{4\pi}
			\arctan\frac{2x}{1-r^2}\\
			&-\frac{xy(1-r^{-4})}{4\pi}
			\log\frac{x^2+(1+y)^2}{x^2+(1-y)^2}
			+\frac{x(1-r^{-2})}{2\pi},
		\end{aligned}
		\]
		where $r^2=x^2+y^2.$ This is \eqref{eq:half-disk-closed}.
	\end{proof}

\section*{Acknowledgment}
 Research of Qinfeng Li was supported by National Natural Science Fund of China for Excellent Young Scholars (No. 12522109), National Key R\&D Program of China (No. 2022YFA1006900) and the National Science Fund of China General Program (No. 12471105). Research of Shuangquan Xie is supported by the National Science fund of China for Youth Scholars (No.12401198) and the Changsha Natural Science Foundation (No. KQ2208006).

\section*{Declaration of generative AI and AI-assisted technologies}

In the course of this work, the authors used ChatGPT (OpenAI) to assist with working out and refining technical details of the proofs, as well as with language editing and LaTeX typesetting. The research questions, proof frameworks, and main mathematical ideas originated with the authors. AI assistance contributed to the implementation of these ideas through interactive discussions with the authors. All AI-assisted arguments were critically examined and verified by the authors, who take full responsibility for the mathematical correctness, originality, references, and final content of the article.

\end{document}